\documentclass[12pt]{amsart}

\usepackage{graphicx}
\usepackage{amsmath}
\usepackage{mathtools}
\usepackage{amssymb}
\usepackage{amsfonts}
\usepackage{amsthm}
\usepackage{enumitem}
\usepackage{color}
\usepackage[pagebackref,colorlinks=true,linkcolor=blue,citecolor=blue]{hyperref}

\newcommand{\Z}{\mathbb{Z}}
\newcommand{\R}{\mathbb{R}}
\newcommand{\BC}{\mathbb{C}}

\newcommand{\SL}{\mathrm{SL}}
\newcommand{\GL}{\mathrm{GL}}
\newcommand{\SO}{\mathrm{SO}}
\newcommand{\Sp}{\mathrm{Sp}}
\newcommand{\St}{\mathrm{St}}

\newcommand{\Ind}{\mathrm{Ind}}
\DeclareMathOperator{\soc}{soc}

\newcommand{\abs}{\lvert\cdot\rvert}
\newcommand{\seg}[3]{\Delta_{#1}[#2,#3]}
\newcommand{\Dodd}{\mathcal{D}_{\mathrm{odd}}}
\newcommand{\Deven}{\mathcal{D}_{\mathrm{even}}}
\newcommand{\ptemp}{\pi_{\mathrm{temp}}}
\newcommand{\one}{\mathbf{1}}
\newcommand{\half}[1]{\frac{#1}{2}}

\makeatletter
\newcommand{\bigtimes}{\DOTSB\mathop{\mathpalette\bigtimes@\relax}\slimits@}
\newcommand{\bigtimes@}[2]{\vcenter{\hbox{$\m@th\ifx#1\displaystyle\scalebox{1.9}{$#1\times$}\else\scalebox{1.35}{$#1\times$}\fi$}}}
\makeatother

\newtheorem{thm}{Theorem}[section]

\newtheorem{lemma}[thm]{Lemma}
\newtheorem{prop}[thm]{Proposition}

\newtheorem{defn}[thm]{Definition}
\newtheorem{notation}[thm]{Notation}
\newtheorem{remark}[thm]{Remark}

\numberwithin{equation}{section}

\let\oldbullet\bullet
\renewcommand{\bullet}{{\vcenter{\hbox{\tiny$\oldbullet$}}}}

\begin{document}
\renewcommand{\theequation}{\arabic{equation}}
\numberwithin{equation}{section}

\title[Purely bad parity unitary dual of $\Sp_{2n}$ and $\SO_{2n+1}$]{On unitary representations of $p$-adic $\Sp_{2n}$ and $\SO_{2n+1}$ of purely bad parity}

\author{Chi-Heng Lo}
\address{Department of Mathematics\\
National University of Singapore\\
119076, Singapore}
\email{ch\_lo@nus.edu.sg}


\begin{abstract}
Let $G$ be $\Sp_{2n}(F)$ or split $\SO_{2n+1}(F)$ over a non-archimedean local field $F$ of characteristic zero. We classify the irreducible unitary representations of $G$ of purely bad parity: such a representation is unitary if and only if it is of Arthur type, except for one additional family when $G=\Sp_{2n}(F)$. In particular, we confirm the remaining case of Ciubotaru's conjecture for $\SO_{2n+1}(F)$, namely that every unitary unipotent representation with integral infinitesimal parameter is of Arthur type.
\end{abstract}

\maketitle

\section{Introduction}\label{sec:intro}

It is a  central problem in representation theory and harmonic analysis to classify the unitary dual of a locally compact group. In this paper, we consider reductive groups over $p$-adic fields. For general linear groups, the unitary dual was classified by Tadi{\'c} (\cite{Tad86}), and for the exceptional group $G_2$ it was determined by Mui{\'c} (\cite{Mui97}). For general split groups, Barbasch and Moy (\cite{BM89}) reduced the unitarity of representations with Iwahori-fixed vectors to a problem for affine Hecke algebras, and the (Iwahori-)spherical unitary dual has been studied extensively by Barbasch and Ciubotaru; see, for example, \cite{Ciu05,BC09,Bar10}.

For classical groups, the complete classification of the unitary dual is still open but much more is known. The unramified unitary dual of split classical groups was classified by Barbasch (\cite{Bar10}) and Mui{\'c}--Tadi{\'c} (\cite{MuT11}), and the generic unitary dual of quasi-split classical groups was determined by Lapid, Mui{\'c} and Tadi{\'c} (\cite{LMT04}). In \cite{Tad23}, Tadi{\'c} classified the unitary representations of corank at most three, and the corank four case was recently settled in \cite{LLW26}. A recent striking result is the classification of the unitary representations of good parity by Atobe and M{\'i}nguez (\cite{AM25}), which we now recall.

Let $F$ be a non-archimedean local field of characteristic zero and let $G=G_n$ be either $\Sp_{2n}(F)$ or the split group $\SO_{2n+1}(F)$. In \cite{AM25}, Atobe and M{\'i}nguez proved that a good parity representation $\pi$ of $G$ is unitary if and only if $\pi$ is of Arthur type. Motivated by this, Ciubotaru (\cite[Conjecture~1.2]{Ciu25}) conjectured that a unipotent representation $\pi$ with \emph{integral} infinitesimal parameter is unitary if and only if $\pi$ is of Arthur type. Although there are counterexamples to this conjecture for exceptional groups (presented by Ciubotaru in his talk \cite{Ciu26}), it is expected to hold for classical groups.

For unipotent representations of $\Sp_{2n}(F)$, the notions of good parity and of integral infinitesimal parameter coincide. 
However, for $\SO_{2n+1}(F)$, a unipotent representation with integral infinitesimal parameter need not be of good parity: such a representation is either of good parity or of ``purely bad parity'', which we now define. Note that this definition is not restricted to unipotent representations.

For an irreducible unitary supercuspidal representation $\rho$ of some $\GL_d(F)$ which is self-dual, let $\alpha_\rho\geq 0$ denote its reducibility point. Namely, $\alpha_{\rho}$ is the unique non-negative real number such that $\rho\abs^{\alpha_{\rho}}\rtimes\one_{G_0}$ is reducible; see \cite{Sil80}. According to \cite[Remark 4.5.2]{MW06}, we know that $\alpha_{\rho} \in \{0, \half{1},1\}$. 

\begin{defn}\label{def:pbp}
An irreducible representation $\pi$ of $G_n$ is of \emph{purely bad parity} if, whenever
\[
\pi\hookrightarrow \rho_1\abs^{x_1}\times\cdots\times\rho_r\abs^{x_r}\rtimes\sigma
\]
with $\rho_i$ unitary supercuspidal, $x_i\in\R$, and $\sigma$ supercuspidal, we have
\begin{enumerate}[label=(\roman*)]
\item $\sigma$ is the trivial representation of $G_0$, and
\item each $\rho_i$ is self-dual and $x_i\in\alpha_{\rho_i}+\tfrac12+\Z$.
\end{enumerate}
\end{defn}

The goal of this paper is to prove the remaining part of the conjecture of Ciubotaru for $\SO_{2n+1}(F)$. Namely, we show that a representation of $\SO_{2n+1}(F)$ of purely bad parity is unitary if and only if it is of Arthur type. We also classify the unitary representations of $\Sp_{2n}(F)$ of purely bad parity. Our main result is the following.

\begin{thm}\label{thm:main}
Suppose that $\pi$ is a representation of $G$ of purely bad parity. Then $\pi$ is unitary if and only if either $\pi$ is of Arthur type, or $G=\Sp_{2n}(F)$ and $\pi\cong \abs^{1/2}\rtimes\pi'$ (always an irreducible induction), for some Arthur type $\pi'$ of $\Sp_{2n-2}(F)$.
\end{thm}

For brevity, we give a name to the condition appearing in the above theorem.

\begin{defn}\label{def A prime}
    We say that a representation $\pi$ of $G_n$ of purely bad parity satisfies condition $(\mathrm{A}')$ if either $\pi$ is of Arthur type, or $G=\Sp_{2n}(F)$ and $\pi\cong \abs^{1/2}\rtimes\pi'$  for some $\pi'$ of Arthur type of $\Sp_{2n-2}(F)$.
\end{defn}

We now explain the main idea of the proof. Let $\pi$ be a representation of $G$ of purely bad parity. We decompose its $L$-parameter $\phi_{\pi}$, which is a $(2n+1)$-dimensional orthogonal (resp.\ $2n$-dimensional symplectic) representation of $W_F \times \SL_2(\BC)$ if $G=\Sp_{2n}(F)$ (resp.\ $G=\SO_{2n+1}(F)$), as
\[ \phi_{\pi}= \bigoplus_{i \in I} \rho_i\lvert \cdot\rvert^{x_i} \otimes S_{a_i}.\]
Here, $\rho_i$ is self-dual, $x_i \in \R$, and $S_{a_i}$ is the $a_i$-dimensional irreducible representation of $\SL_2(\BC)$. Let
\[ m_{\rho\lvert\cdot \rvert^{x}\otimes S_{a}}(\pi):= |\{i \in I\ | \ \rho_i\lvert \cdot\rvert^{x_i} \otimes S_{a_i}= \rho\lvert\cdot \rvert^{x}\otimes S_{a}\}|\]
denote the corresponding multiplicity. Unlike the good parity case, Arthur type representations of purely bad parity is simple to describe. Indeed, a purely bad parity representation $\pi$ is of Arthur type if and only if the following two conditions both hold.
\begin{itemize}
    \item [(a)]For any $\rho\lvert\cdot \rvert^{x}\otimes S_{a}$, the multiplicity $m_{\rho\lvert\cdot \rvert^{x}\otimes S_{a}}(\pi)$ is even, unless $\rho\lvert\cdot \rvert^{x}\otimes S_{a}$ is the trivial representation and $G= \Sp_{2n}(F)$.
    \item [(b)] For any $\rho\lvert\cdot \rvert^{x}\otimes S_{a}$ such that $x<0$, we have
    \[ m_{\rho\lvert\cdot \rvert^{x}\otimes S_{a}}(\pi) \leq m_{\rho\lvert\cdot \rvert^{x+1}\otimes S_{a}}(\pi). \]
\end{itemize}
This follows from M{\oe}glin's result \cite[\S 4.2]{Moe11} that Arthur type representations of purely bad parity are all irreducibly induced from products of generalized Speh representations (see Theorem \ref{thm:Moeglin} and \S \ref{subsec:segments} for the notation)
\[\pi_A= \bigtimes_{j \in J } u_{\rho_j}(a_j, b_j) \rtimes \pmb{1},\]
where $\pmb{1}$ is the trivial representation of $\Sp_{0}(F)$ or $\SO_{1}(F)$. Each generalized Speh representation $ u_{\rho_j}(a_j, b_j)$ contributes two copies of
\[ \rho_j\lvert \cdot\rvert^{\half{1-b_j}}\otimes S_{a_j} \oplus  \rho_j\lvert \cdot\rvert^{\half{3-b_j}}\otimes S_{a_j}   \oplus \cdots \oplus  \rho_j\lvert \cdot\rvert^{\half{b_j-1}}\otimes S_{a_j} \]
 to the $L$-parameter $\phi_{\pi_A}$, and the observation follows. (Here we regard $\phi_{\pmb{1}}$ of $\Sp_{0}(F)$ as the trivial representation of $W_F \times \SL_2(\BC)$.)

 Now suppose that $\pi$ is unitary but not of Arthur type. Further assume that 
 \[m_{\rho\lvert\cdot \rvert^{\half{1-b}}\otimes S_{a}}(\pi) - m_{\rho\lvert\cdot \rvert^{\half{3-b}}\otimes S_{a}}(\pi)\]
 is a positive even integer $2r$. The main idea of our argument is the following ``Arthurization'' process of the $L$-parameter: Set 
 \[ u_{\rho\lvert\cdot \rvert^{\half{1-b}}\otimes S_{a}}:= u_{\rho}(a,b),\ \  \widetilde{u}_{\rho\lvert\cdot \rvert^{\half{1-b}}\otimes S_{a}}:= u_{\rho}(a,b-2). \]
 Since $\pi$ is unitary, the unitary induction $\widetilde{u}_{\rho\lvert\cdot \rvert^{\half{1-b}}\otimes S_{a}} \rtimes \pi$ is semisimple and all irreducible subquotients are unitary. In particular, the Langlands subquotient $\pi'$ of $(\widetilde{u}_{\rho\lvert\cdot \rvert^{\half{1-b}}\otimes S_{a}})^r \rtimes \pi$ (see the end of \S \ref{subsec:Langlands} for definition), which has $L$-parameter
 \[ \phi_{\pi'}= \phi_{\pi}+ ( \rho\lvert \cdot\rvert^{\half{3-b}}\otimes S_a \oplus  \rho\lvert \cdot\rvert^{\half{5-b}}\otimes S_a   \oplus \cdots \oplus \rho\lvert \cdot\rvert^{\half{b-3}}\otimes S_a )^{\oplus 2r},\]
 is unitary. Then note that
 \[m_{\rho\lvert\cdot \rvert^{\half{1-b}}\otimes S_{a}}(\pi') - m_{\rho\lvert\cdot \rvert^{\half{3-b}}\otimes S_{a}}(\pi')=0. \]
Thus $\pi'$ is closer to being of Arthur type. Indeed, if $\pi$ satisfies condition (a), then repeating this process finitely many times produces a representation of Arthur type.

Now we outline the proof of Theorem \ref{thm:main}. 
\begin{enumerate}
    \item [(1)] In the first step, we show that if $\pi$ is a unitary representation of purely bad parity, then condition (a) above holds, except that when $G_n=\Sp_{2n}(F)$, $m_{\lvert\cdot \rvert^{1/2} \otimes S_1}(\pi)$ can be odd; see Proposition \ref{prop:parity}. This exception leads to the condition $(\mathrm{A}')$ in Definition \ref{def A prime}. 

    We sketch the proof of this proposition. Suppose that a unitary representation $\pi$ violates this condition. Then, via a slight variant of the above Arthurization process, we obtain a unitary representation $\pi''$, which is the Langlands subquotient of an induction $\Pi_0$ where
    \[ \Pi_s= \bigtimes_{i \in I_{odd}} u_{\rho_i}(a_i,b_i)\lvert\cdot \rvert^{-1/2+s} \times \bigtimes_{i \in I_{even}} u_{\rho_i}(a_i,b_i) \rtimes \pmb{1}. \]
    However, applying \cite[Theorem 1.2]{BS25} and \cite[Theorem 5.3]{Ato22b}, the induction $\Pi_s$ is irreducible for all $ s \in [0, \half{1})$. In particular, we have $\pi''=\Pi_0$. By the classification of unitary complementary Arthur representations (\cite[Theorem~1.3]{HJLLZ25} or \cite[Theorem 4.1]{AM25}), the representations $\Pi_{s}$ with $s \in (0,\half{1})$ are unitary if and only if $u_{\rho_i}(a_i,b_i)\rtimes\pmb{1}$ is irreducible for every $i\in I_{odd}$, which in turn is equivalent to the irreducibility of the unitary induction $\Pi_{\half{1}}$. Finally, using \cite[Theorem~4.4]{Ato22b} and the non-vanishing criterion of \cite{Ato22a}, we show that that $\Pi_{\half{1}}$ is irreducible if and only if $I_{odd}$ is empty, or $G_n=\Sp_{2n}(F)$ and  
    \[ \bigtimes_{i \in I_{odd}} u_{\rho_i}(a_i,b_i)= \lvert\cdot \rvert^{0}.\]
    This gives the desired conclusion.
    \item [(2)] In the second step, we reduce the theorem to a special case; see Proposition \ref{prop:reduction}. Suppose that there exists a unitary representation $\pi$ of purely bad parity but condition $(\mathrm{A}')$ fails. Using the control of $\phi_{\pi}$ obtained in Step (1), we may repeatedly apply the Arthurization process to construct a sequence of unitary representations,
    \[ \pi=\pi^1 \rightsquigarrow \pi^2 \rightsquigarrow \cdots  \rightsquigarrow \pi^{r-1} \rightsquigarrow  \pi^r,\]
    where the last one, $\pi^r$, satisfies condition $(\mathrm{A}')$.    
    We look at the representation $\pi''=\pi^{r-1}$, which is one step away from satisfying condition $(\mathrm{A}')$. If we apply the Arthurization in a suitable order, we can assume that the inequality in condition (b) for $\pi''$ holds for all $\rho\lvert\cdot \rvert^{x}\otimes S_{a}$ such that $x<0$, except for exactly one $\rho_0\lvert\cdot \rvert^{x_0}\otimes S_{a_0}$, for which
 \[ m_{\rho_0\lvert\cdot \rvert^{x_0}\otimes S_{a_0}}(\pi'') = m_{\rho_0\lvert\cdot \rvert^{x_0+1}\otimes S_{a_0}}(\pi'')+2. \]
We call a representation of this form a \emph{minimal counter example} (see Definition \ref{def:reduced} for a precise definition), since it is one step away from satisfying condition $(\mathrm{A}')$. 
\item[(3)] In the final step, we show that no minimal counter example is unitary; see Theorem \ref{thm:reduced}. 

Let $\pi$ be a minimal counter example, and let $\Delta:=\St(\rho_0,a_0)\abs^{x_0}$ be the segment above. 
Applying the Arthurization process to $\pi$ one more time, we obtain a representation $\tau$ satisfying condition $(\mathrm{A}')$ that occurs as a subquotient of the induction $\widetilde{u}_\Delta \rtimes \pi$. By \ref{U1}, to show that $\pi$ is not unitary, it suffices to show that $\tau$ cannot occur as a subrepresentation of  $\widetilde{u}_\Delta \rtimes \pi$. To achieve this, we design a composition of derivatives $D$ (see \S \ref{subsec:derivatives}) case by case and verify that
\begin{itemize}
    \item any subrepresentation $\sigma$ of $\widetilde{u}_\Delta \rtimes \pi$ satisfies $D(\sigma) \neq 0$ by Frobenius reciprocity, and
    \item we have  $D(\tau)=0$ by the Leibniz rule for derivatives.
\end{itemize}
This completes the proof of Theorem \ref{thm:main}.
\end{enumerate}

We remark that our method mostly relies on Tadi{\'c}'s formula and Frobenius reciprocity, which are available for other classical groups as well. Besides these, we use the following inputs:
\begin{enumerate}
    \item irreducibility criterion for parabolic inductions from essentially generalized Speh representations and a representation of Arthur type \cite{Ato22b, BS25};
    \item the classification of unitary complementary Arthur representations ; see \cite[Theorem 1.3]{HJLLZ25} and \cite[Theorem 4.1]{AM25};
    \item the irreducibility criterion for the good parity unitary induction $u_{\rho}(a,b)\rtimes\pmb{1}$ from the Siegel parabolic subgroup, which amounts to deciding whether the corresponding local Arthur packet is a singleton.
\end{enumerate}
Once these results are available for other classical groups, our argument should apply with minor changes.

\subsection*{Structure of the paper}
In \S\ref{sec:prelim}, we recall the Langlands classification and the local Langlands correspondence, derivatives, and the facts on unitarity that we need. In \S\ref{subsec:Speh}, we study generalized Speh representations: we prove a constraint on their consecutive derivatives (Lemma~\ref{lem:consecutive der Speh}), collect consequences of Tadi{\'c}'s irreducibility criterion, recall the results of M{\oe}glin, Bo\v{s}njak--Stadler and Atobe, and prove the easier direction of Theorem~\ref{thm:main} (Proposition~\ref{prop:A prime unitary}). In \S\ref{sec:reduction}, we characterize condition $(\mathrm{A}')$ in terms of multiplicities, prove the parity constraint (Proposition~\ref{prop:parity}), and reduce Theorem~\ref{thm:main} to minimal counter examples (Proposition~\ref{prop:reduction}). In \S\ref{sec:main}, we show that minimal counter examples are not unitary (Theorem~\ref{thm:reduced}). Finally, we complete the proof of Theorem~\ref{thm:main} in \S\ref{sec:proof}.

\subsection*{Acknowledgments}
The author would like to thank Baiying Liu for his constant support, and Alexander Hazeltine, Alberto M{\'i}nguez and Brian Wen for helpful discussions. Some of the ideas in this paper were inspired by talks and discussions at the conference “Recent Progress on the Unitary Dual Problem” at the Brin Mathematics Research Center, and the author is grateful to the organizers.

\subsection*{Declaration on the use of generative AI}
During the preparation of this paper, the author used Claude (Anthropic), an AI assistant, to typeset handwritten notes into \LaTeX, to search the literature, and to polish the exposition. All mathematical statements and proofs were developed and verified by the author, who takes full responsibility for the content of the paper.

\section{Preliminaries}\label{sec:prelim}

\subsection{Notation}\label{subsec:notation}

Throughout, $F$ is a non-archimedean local field of characteristic zero and $\abs$ denotes its normalized absolute value, as well as the character $\abs\circ\det$ of $\GL_k(F)$.
We write $G_n$ for $\Sp_{2n}(F)$ or split $\SO_{2n+1}(F)$, and $\one=\one_{G_0}$ for the trivial representation of $G_0$, which is $\Sp_0(F)$ or $\SO_1(F)$.  All representations are smooth and of finite length. We let $\Pi(G)$ denote the set of isomorphism classes of irreducible representations of $G$, for $G=G_n$ or $\GL_n(F)$.

We fix a Borel subgroup of $G=G_n$ or $\GL_n(F)$, which determines a collection of standard parabolic subgroups. These standard parabolic subgroups are in bijection with ordered partitions of $n$ (with $k_r\geq 0$ allowed in the case of $G_n$), so we write them as $P_{(k_1, \ldots ,k_r)}$ with $\sum_{i} k_i=n$. The Levi subgroup $M_{(k_1, \ldots ,k_r)}$ is isomorphic to
\[ \begin{cases}
    \GL_{k_1}(F) \times \cdots \times \GL_{k_r}(F) & \text{if }G= \GL_n(F),\\
    \GL_{k_1}(F) \times \cdots \times \GL_{k_{r-1}}(F) \times G_{k_r} & \text{ if }G=G_n.
\end{cases}  \]
For representations $\tau_i$ of $\GL_{k_i}(F)$ and $\pi$ of $G_m$, we regard
\[\tau_1\otimes\cdots\otimes\tau_r, \quad\text{and}\quad \tau_1\otimes\cdots\otimes\tau_r\otimes\pi  \]
as representations of the standard Levi subgroups $M_{(k_1, \ldots ,k_r)}$ of $\GL_{k_1+\cdots+k_r}(F)$ or $M_{(k_1, \ldots ,k_r, m)}$ of $G_{k_1+\cdots+k_r+m}$, respectively. Then let
\[
\tau_1\times\cdots\times\tau_r \quad\text{and}\quad \tau_1\times\cdots\times\tau_r\rtimes\pi
\]
denote the normalized parabolic inductions from the corresponding standard parabolic subgroups to $\GL_{k_1+\cdots+k_r}(F)$ and $G_{k_1+\cdots+k_r+m}$, respectively.

For a representation $\tau$ of a general linear group and $k\geq 0$ we write $\tau^k=\tau\times\cdots\times\tau$ ($k$ times). We use ``$=$'' between induced representations to mean an isomorphism of finite length representations.

\subsection{Segments}\label{subsec:segments}

Let $\rho$ be an irreducible unitary supercuspidal representation of $\GL_d(F)$. Any essentially square-integrable representation of a general linear group is of the form
\[
\Delta=\St(\rho,a)\abs^{s},\qquad a\in\Z_{>0},\ s\in\R,
\]
where $\St(\rho,a)$ is the Steinberg representation, i.e.\ the unique irreducible subrepresentation of $\rho\abs^{\frac{a-1}{2}}\times\rho\abs^{\frac{a-3}{2}}\times\cdots\times\rho\abs^{\frac{1-a}{2}}$. We set 
\[
e(\Delta):=s,\qquad x(\Delta):=\frac{a-1}{2}+s,\qquad y(\Delta):=\frac{1-a}{2}+s,
\]
and we also write $\Delta=\seg{\rho}{x(\Delta)}{y(\Delta)}$. Thus $\seg{\rho}{x}{y}$ is defined whenever $x-y\in\Z_{\geq0}$, and it is the unique irreducible subrepresentation of $\rho\abs^{x}\times\rho\abs^{x-1}\times\cdots\times\rho\abs^{y}$. We refer to such $\Delta$ as a \emph{segment}. 

We sometimes identify $\seg{\rho}{x}{y}$ with the set of supercuspidal representations \[\{\rho\abs^x,\rho\abs^{x-1},\dots,\rho\abs^y\},\] and we use the convention that $\seg{\rho}{x}{x+1}$ is the trivial representation of $\GL_0(F)$. The contragredient of $\Delta$ is denoted by $\Delta^\vee=\seg{\rho^\vee}{-y}{-x}$, and $\Delta\abs^{t}=\seg{\rho}{x+t}{y+t}$.

Two segments $\Delta=\seg{\rho}{x}{y}$ and $\Delta'=\seg{\rho'}{x'}{y'}$ are \emph{linked} if $\rho\cong\rho'$, $x-x'\in\Z$, neither of $\Delta,\Delta'$ contains the other, and $\Delta\cup\Delta'$ is again a segment. By \cite{Zel80}, $\Delta\times\Delta'$ is irreducible if and only if $\Delta$ and $\Delta'$ are not linked. 

\subsection{Langlands classification and local Langlands correspondence}\label{subsec:Langlands}
In this subsection, we recall the  subrepresentation form of the Langlands classification (\cite{BW00,Sil78,Kon03}), and the local Langlands correspondence for general linear groups and for $G_n$.

For general linear groups, the Langlands classification states that for every irreducible representation $\pi$, there exist segments $\Delta_1,\ldots,\Delta_f$ with $e(\Delta_1)\leq\cdots\leq e(\Delta_f)$ such that
\[ \pi \hookrightarrow \Delta_1\times\cdots\times\Delta_f\]
as the unique irreducible subrepresentation. The multi-set $\{\Delta_1,\ldots, \Delta_f\}$ is uniquely determined by $\pi$ \cite{Zel80}, and we call the right hand side above the \emph{standard module} of $\pi$, denoted by $M(\pi)$, and write 
\[\pi= L(\Delta_1,\dots,\Delta_f).\]

The local Langlands correspondence for general linear groups (\cite{HT01,Hen00,Sch13}) gives a bijection between $\Pi(\GL_n(F))$ and the set of $L$-parameters of $\GL_n(F)$, that is, isomorphism classes of $n$-dimensional representations of the Weil--Deligne group $W_F \times \SL_2(\BC)$.
Under this correspondence, a unitary supercuspidal representation $\rho$ of $\GL_d(F)$ corresponds to a $d$-dimensional irreducible unitary representation of $W_F$, which we again denote by $\rho$ by abuse of notation. The remaining cases are described in terms of this correspondence for supercuspidal representations. For an essentially square-integrable representation $\Delta= \St(\rho,a)\abs^s$, the $L$-parameter is given by $\phi_{\Delta}:= \rho \abs^s\otimes S_{a}$, where $S_a$ is the unique $a$-dimensional irreducible representation of $\SL_2(\BC)$. For a general representation 
\[\pi= L(\Delta_1,\dots,\Delta_f),\]
the $L$-parameter is given by
\[ \phi_{\pi}= \phi_{\Delta_1} + \cdots + \phi_{\Delta_f}.\]

Now we turn to the classical group $G_n$. The Langlands classification states that for any $\pi \in \Pi(G_n)$, there exist
\begin{itemize}
    \item a collection of segments $\Delta_1,\dots,\Delta_f$ with $
e(\Delta_1)\leq\cdots\leq e(\Delta_f)<0$, and
\item a tempered representation $\ptemp \in \Pi(G_m)$ with $m \leq n$, such that
\end{itemize}
\[ \pi \hookrightarrow \Delta_1\times\cdots\times\Delta_f\rtimes\ptemp\]
as the unique irreducible subrepresentation. The data $(\{\Delta_1,\dots,\Delta_f\},\ptemp)$, called the $L$-data of $\pi$, are uniquely determined by $\pi$. We call the right hand side above the \emph{standard module} of $\pi$, denoted by $M(\pi)$, and write
\[ \pi=L(\Delta_1,\dots,\Delta_f;\ptemp).\]

The local Langlands correspondence from \cite{Art13} assigns to each $\pi \in \Pi(G_n)$ an $L$-parameter $\phi_{\pi}$ of $G_n$, which we identify with an isomorphism class of 
\begin{itemize}
\item $2n$-dimensional symplectic representations of $W_F \times \SL_2(\BC)$ if $G_n=\SO_{2n+1}(F)$, or
    \item $(2n+1)$-dimensional orthogonal representations of $W_F \times \SL_2(\BC)$ if $G_n=\Sp_{2n}(F)$.
\end{itemize}
In terms of the Langlands classification, the $L$-parameters of $\pi$ and $\ptemp$ are related by
\[ \phi_{\pi}= \phi_{\Delta_1}+ \cdots + \phi_{\Delta_f} + \phi_{\ptemp} +  \phi_{\Delta_f}^{\vee}+ \cdots + \phi_{\Delta_1}^{\vee}.\]

The $L$-packet $\Pi_{\phi}(G):= \{\pi \in \Pi(G) \ | \ \phi_{\pi}=\phi \}$ is in bijection with the set of irreducible characters of the component group 
\[ \mathcal{S}_{\phi}:= Z_{\widehat{G}}(\phi) /Z_{\widehat{G}}(\phi)^0 Z(\widehat{G})^{\Gamma}.\]

Now we specialize to the case that
\[ \pi=L(\Delta_1,\dots,\Delta_f;\ptemp)\]
is of purely bad parity. In this case, we have a more precise description of the $L$-parameter. Since
the tempered representation $\ptemp$ is also of purely bad parity, it is an irreducible induction from a product of Steinberg representations (see Theorem \ref{thm:Moeglin} below)
\[ \ptemp= \St(\rho_1,a_1) \times \cdots \times \St(\rho_k, a_k) \rtimes \pmb{1}.\]
Then we have 
\[ \phi_{\ptemp}= (\rho_1 \otimes S_{a_1} +\cdots + \rho_k\otimes S_{a_k})^{\oplus 2} + \phi_{\pmb{1}}. \]
Here $\phi_{\pmb{1}}$ is the zero representation if $G_{n}= \SO_{2n+1}(F)$, and the trivial representation if $G_n=\Sp_{2n}(F)$. Consequently, $\mathcal{S}_{\phi_{\pi}}$ is trivial, and hence the $L$-packet $ \Pi_{\phi_{\pi}}(G)=\{\pi\}$ is a singleton. In other words, there is a bijection
    \[ \{\text{purely bad parity representations of }G_n\} \longleftrightarrow \{\text{purely bad parity }L\text{-parameters of }G_n\}.\]

We end this subsection by introducing the notion of \emph{Langlands subquotient}.
Suppose that we have a parabolic induction of the form
\[ \Pi= \tau_1 \times \cdots \times \tau_f \rtimes \sigma,\]
where $\tau_i$ and $\sigma$ are irreducible, not necessarily essentially square-integrable. Then there exists at least one irreducible subquotient of $\Pi$ that lies in the $L$-packet of 
\[ \phi_{\Pi}:=\phi_{\tau_1}+ \cdots + \phi_{\tau_f} + \phi_{\sigma} +  \phi_{\tau_f}^{\vee}+ \cdots + \phi_{\tau_1}^{\vee},  \]
where $\phi_{\tau_i}$ and $\phi_{\sigma}$ are the $L$-parameters of $\tau_i$ and $\sigma$ respectively; see \cite[Lemma~10.3]{HLLS24} for example. As explained above, in the purely bad parity setting, the $L$-packet $\Pi_{\phi_{\Pi}}$ is a singleton. Therefore, we call the unique member of $\Pi_{\phi_{\Pi}}$ the \emph{Langlands subquotient} of the induction $\Pi$.

\subsection{Derivatives}\label{subsec:derivatives}

We recall the notion of $\rho\abs^z$-derivatives of Jantzen \cite{Jan18} (see also \cite{AM23}). Let $\rho\abs^z$ be a supercuspidal representation of $\GL_d(F)$ with $\rho$ self-dual and $z\in\R$. For a representation $\pi$ of $G_n$ and an integer $1\leq k\leq n/d$, recall that $P_{(kd, n-kd)}$ is the standard parabolic subgroup of $G_n$ with Levi subgroup isomorphic to $\GL_{kd}(F) \times G_{n-kd}$. We write the semisimplified Jacquet module as 
\[ [\textrm{Jac}_{P_{(kd, n-kd)}} (\pi)]= (\rho\abs^z)^k \otimes D^{(k)}_{\rho\abs^z}(\pi)+ \sum_{i\in I} \tau_i \otimes \pi_i,\]
where $\tau_i \not\cong (\rho\abs^z)^k$ for all $i \in I$. The semisimple representation $D^{(k)}_{\rho\abs^z}(\pi)$ (possibly zero) is called the $k$-th derivative of $\pi$ with respect to $\rho\abs^z$. We say that $D^{(k)}_{\rho\abs^z}(\pi)$ is the highest derivative of $\pi$ if it is nonzero but $D^{(k+1)}_{\rho\abs^z}(\pi)=0$, and that $\pi$ is $\rho\abs^{z}$-reduced if $D^{(1)}_{\rho\abs^z}(\pi)=0$.

In the case of general linear groups, we have left and right derivatives defined similarly: 
\begin{align*}
   [\textrm{Jac}_{P_{(kd, n-kd)}} (\pi)]&= (\rho\abs^z)^k \otimes L^{(k)}_{\rho\abs^z}(\pi)+ \sum_{i\in I_1} \tau_i \otimes \pi_i,\\ 
  [\textrm{Jac}_{P_{(n-kd, kd)}} (\pi)]&=   R^{(k)}_{\rho\abs^z}(\pi) \otimes (\rho\abs^z)^k + \sum_{i\in I_2} \pi_i \otimes \tau_i, 
\end{align*}
where $\tau_i \not\cong (\rho\abs^z)^k$ for all $i \in I_1\sqcup I_2$. Highest left and right derivatives, and left and right $\rho\abs^{z}$-reducedness, are defined in the same way.

We write $D_{\rho\abs^z}:=D_{\rho\abs^z}^{(1)}$, $L_{\rho\abs^z}:=L_{\rho\abs^z}^{(1)}$, and $R_{\rho\abs^z}:=R_{\rho\abs^z}^{(1)}$ for simplicity. Note that
\[D_{\rho\abs^z}^{(k)}= \frac{1}{k!} \underbrace{D_{\rho\abs^z} \circ  \cdots \circ D_{\rho\abs^z}}_{k \text{ times}}. \]

Suppose that $\pi$ is irreducible, and that $z \neq 0$ in the case of classical groups. Then the highest derivative of $\pi$ with respect to $\rho\abs^z$ is irreducible (see \cite[\S3.2]{AM23} and the references therein). Explicit formulas are given in \cite[Theorem~5.11]{LM16} for general linear groups, and in \cite[Proposition~6.1, Theorems~7.1 and 7.4]{AM23} for $G_n$. In this paper, we do not use these formulas. Instead, we only use the following properties of these derivatives.
\begin{enumerate}
    \item [(i)] Frobenius reciprocity: If 
    \[ \pi \hookrightarrow (\rho\abs^{x_1})^{r_1} \times \cdots\times  (\rho\abs^{x_k})^{r_k} \rtimes \Pi\qquad  (\text{resp. } \sigma \hookrightarrow (\rho\abs^{x_1})^{r_1} \times \cdots\times  (\rho\abs^{x_k})^{r_k} \times \Sigma),\]
    for some nonzero representation $\Pi$ of a classical group (resp.\ $\Sigma$ of a general linear group),
    then
    \[ D_{\rho\abs^{x_k}}^{(r_k)}\circ\cdots \circ D_{\rho\abs^{x_1}}^{(r_1)}(\pi) \neq 0,\qquad(\text{resp. }  L_{\rho\abs^{x_k}}^{(r_k)}\circ \cdots \circ L_{\rho\abs^{x_1}}^{(r_1)}(\sigma) \neq 0). \]
    Conversely, if $D_{\rho\abs^x}^{(k)}(\pi) \neq 0$ (resp. $L_{\rho\abs^x}^{(k)}(\sigma) \neq 0$), then there exists an irreducible subquotient $\pi'$ of $D_{\rho\abs^x}^{(k)}(\pi)$ (resp. $\sigma'$ of $L_{\rho\abs^x}^{(k)}(\sigma)$) such that
    \[ \pi \hookrightarrow (\rho\abs^x)^{k} \rtimes \pi'\qquad(\text{resp. }  \sigma \hookrightarrow (\rho\abs^x)^{k} \rtimes \sigma'). \]
    \item [(ii)] Leibniz rule: In the Grothendieck group, we have
\begin{equation}\label{eq:Leibniz}
D^{(k)}_{\rho\abs^z}(\tau\rtimes\pi)=\sum_{k_1+k_2+k_3=k} c_{(k_1,k_2,k_3)} L^{(k_1)}_{\rho\abs^z}\circ R^{(k_2)}_{\rho\abs^{-z}}(\tau)\rtimes D^{(k_3)}_{\rho\abs^z}(\pi),
\end{equation}
where the $c_{(k_1,k_2,k_3)}$ are positive integers. Similarly, for representations $\tau,\tau'$ of general linear groups,
\[ L^{(k)}_{\rho\abs^z}(\tau\times\tau')=\sum_{k_1+k_2=k}c_{(k_1,k_2)}L^{(k_1)}_{\rho\abs^z}(\tau)\times L^{(k_2)}_{\rho\abs^z}(\tau')\]
for some positive integers $c_{(k_1,k_2)}$, and likewise for right derivatives.
\item[(iii)] Derivatives of Steinberg representations:
\begin{align}\label{eq der St}
    L_{\rho\abs^{z}}(\Delta_{\rho}[x,y])= \begin{cases}  
\Delta_{\rho}[x-1,y] & \text{ if } z=x, \\
0 & \text{ otherwise,}
\end{cases}\ \ \ \ R_{\rho\abs^{z}}(\Delta_{\rho}[x,y])= \begin{cases}  
    \Delta_{\rho}[x,y+1] & \text{ if } z=y, \\
    0 & \text{ otherwise.}
    \end{cases}
\end{align}
\end{enumerate}
We remark that the Leibniz rule is a consequence of Tadi{\'c}'s formula \cite{Tad95} (and of the geometric lemma \cite{Zel80} for general linear groups).

\subsection{Unitarity}\label{subsec:unitarity}

In this paper, we only use the following standard facts about unitarity.
\begin{enumerate}[label=(U\arabic*)]
\item\label{U1} If $\tau$ is an irreducible unitary representation of a general linear group and $\pi$ is irreducible unitary, then $\tau\rtimes\pi$ is unitary, hence semisimple; in particular every irreducible subquotient of $\tau\rtimes\pi$ is a subrepresentation, and is unitary.
\item\label{U2} (Complementary series) Let $\{\Pi_s\}_{s\in[s_0,s_1)}$ be a continuous family of irreducible Hermitian representations obtained by varying the exponent of an induction. If $\Pi_{s_0}$ is unitary, then so is $\Pi_s$ for all $s\in[s_0,s_1)$, and every irreducible subquotient of the limit $\Pi_{s_1}$ is unitary.
\item\label{U3} The generalized Speh representations $u_{\rho}(a,b)$ (defined in \eqref{eq:def Speh} below) with $\rho$ unitary are unitary \cite{Tad86}.
\end{enumerate}

\section{Generalized Speh representations}\label{subsec:Speh}

A \emph{generalized Speh representation}  $u$ of a general linear group is an irreducible representation of the form
\begin{align}\label{eq:def Speh}
    u=u_\rho(a,b):=L\big(\St(\rho,a)\abs^{\frac{1-b}{2}},\ \St(\rho,a)\abs^{\frac{3-b}{2}},\ \dots,\ \St(\rho,a)\abs^{\frac{b-1}{2}}\big),
\end{align}
where $\rho$ is a unitary supercuspidal representation of a general linear group and $a,b\in\Z_{>0}$.
We call $\St(\rho,a)\abs^{\frac{1-b}{2}}$ (resp.\ $\St(\rho,a)\abs^{\frac{b-1}{2}}$) the lowest (resp.\ highest) segment of $u$.  A twist $u_\rho(a,b)\abs^t$ is called an \emph{essentially generalized Speh representation}.

\begin{notation}\label{not:UDelta}\ 
\begin{enumerate}
    \item For a generalized Speh representation $u=u_\rho(a,b)$, we set $\widetilde{u}:=u_\rho(a,b-2)$, which is obtained from $u$ by removing its lowest and highest segments.
    \item For a segment of the form $\Delta=\St(\rho,a)\abs^{\frac{1-b}{2}}$ with $b\in\Z_{\geq1}$ and $\rho$ self-dual, set
\[
u_\Delta:=u_\rho(a,b),\qquad \widetilde{u}_\Delta:=u_\rho(a,b-2),
\]
where $\widetilde{u}_\Delta$ is the trivial representation of $\GL_0(F)$ if $b\leq 2$. Thus $\Delta$ is the lowest segment of $u_\Delta$, and $\widetilde{u}_\Delta=u_{\Delta\abs^{1}}$ when $b\geq 3$.
\end{enumerate}

\end{notation}

\begin{lemma}\label{lem:socle}
Let $\Delta=\St(\rho,a)\abs^{\frac{1-b}{2}}$ with $b\geq2$ and $\rho$ self-dual. Then
\[
\soc\big(\Delta\times {\widetilde{u}_\Delta}\times\Delta^\vee\big)=u_\Delta,\qquad \soc\big(\Delta\times {\widetilde{u}_\Delta}\big)=u_\rho(a,b-1)\abs^{-\frac12}.
\]
\end{lemma}

\begin{proof}
We have injections
\begin{align*}
    \Delta\times {\widetilde{u}_\Delta}\times\Delta^\vee \hookrightarrow \Delta\times M({\widetilde{u}_\Delta})\times\Delta^\vee= M(u_{\Delta}),
\end{align*}
and
\begin{align*}
   \Delta\times {\widetilde{u}_\Delta} \hookrightarrow \Delta\times M({\widetilde{u}_\Delta})= M(u_\rho(a,b-1)\abs^{-\frac12}).
\end{align*}
Since the right hand sides have a unique irreducible subrepresentation $u_{\Delta}$ and $u_\rho(a,b-1)\abs^{-\frac12}$ respectively, the conclusion follows.
\end{proof}

\subsection{Consecutive derivatives of generalized Speh representations}
In this subsection, we prove the following constraint on consecutive derivatives of generalized Speh representations.
\begin{lemma}\label{lem:consecutive der Speh}
    Let $u= u_{\rho}(a,b)$ be a generalized Speh representation, and write the lowest segment of $u$ as  
    \[ \Delta_{\rho}[\alpha, \beta]:= \Delta_{\rho}[ \half{a-1}+\half{1-b}, \half{1-a}+ \half{1-b}].\]    
    Suppose that 
    \[  L_{\rho\abs^{z_r}} \circ \cdots \circ L_{\rho\abs^{z_1}} (u) \neq 0 \ \ \ (\text{resp.  } \ \ R_{\rho\abs^{-z_r}} \circ \cdots \circ {R}_{\rho\abs^{-z_1}} (u) \neq 0 ),\]
    where $z_1 \geq \cdots \geq z_r$. Then we must have 
    \begin{itemize}
        \item [(a)]$z_1= \alpha$, and
        \item [(b)]$z_i= z_1-i+1$ for $1 \leq i \leq r$, and
        \item [(c)]$z_r \geq \beta $.
    \end{itemize}     
\end{lemma}
\begin{proof}
We show the assertion for the left derivatives, and omit the verification for the right derivatives since it is similar.
    Part (a) follows from the formula for Jacquet modules of ladder representations \cite{KL12}, since $u$ is a ladder representation. For Part (b),   
    we have an injection
    \[ u \hookrightarrow \Delta_{\rho}[\alpha, \beta] \times u_\rho(a,b-1)\abs^{1/2}.\]
    By Part (a), we see that $u_\rho(a,b-1)\abs^{1/2}$ is left $\rho\abs^{z}$-reduced unless $z= z_1+1$. Since $z_i \leq z_1< z_1+1$, the Leibniz rule gives
    \begin{align*}
        0 \neq L_{\rho\abs^{z_r}} \circ \cdots \circ L_{\rho\abs^{z_1}} (u) & \leq L_{\rho\abs^{z_r}} \circ \cdots \circ L_{\rho\abs^{z_1}} (\Delta_{\rho}[\alpha, \beta] \times u_\rho(a,b-1)\abs^{1/2}) \\
        &=  L_{\rho\abs^{z_r}} \circ \cdots \circ L_{\rho\abs^{z_1}} (\Delta_{\rho}[\alpha, \beta]) \times u_\rho(a,b-1)\abs^{1/2}.
    \end{align*}
Then the desired assertion follows from applying \eqref{eq der St} repeatedly.    
\end{proof}

\subsection{Product of essentially generalized Speh representations}

To an essentially generalized Speh representation $u=u_{\rho}(a,b)\abs^s$, we associate a quadruple
\[q(u)=q(u_{\rho}(a,b)\abs^s )=(-A+s,B+s,-B+s,A+s), \]
where $A=\frac{a+b}{2}-1$ and $B=\frac{a-b}{2}$. For quadruples we write $q<_{\mathrm{strong}}q'$ if each entry of $q$ is strictly smaller than the corresponding entry of $q'$, and for essentially Speh representations $u,u'$ we write $u<_{\mathrm{strong}}u'$ if $q(u)<_{\mathrm{strong}}q(u')$. We use the following irreducibility criterion of Tadi\'c.

\begin{thm}[{\cite[Theorem~1.1]{Tad14}}]\label{thm:Tadic}
Let $u=u_\rho(a,b)\abs^s$ and $u'=u_{\rho'}(a',b')\abs^{s'}$ be essentially Speh representations. Then $u\times u'$ is reducible if and only if
\begin{enumerate}[label=(\roman*)]
\item $\rho\cong\rho'$, the exponents of $u$ and $u'$ lie in the same coset modulo $\Z$, and the union of the supports of $u$ and $u'$ is a segment, and
\item $q(u)<_{\mathrm{strong}}q(u')$ or $q(u')<_{\mathrm{strong}}q(u)$.
\end{enumerate}
If $u\times u'$ is irreducible, then $u\times u'\cong u'\times u$.
\end{thm}

In this subsection, we list some direct consequences of the theorem.

\begin{lemma}\label{lem:Speh-irr}
Let $\rho$ be self-dual and $a,b,c,d\geq1$.
\begin{enumerate}[label=(\alph*)]
\item If $t,t'\in\R$ satisfy $|t-t'|<1$, then $u_\rho(a,b)\abs^{t}\times u_\rho(c,d)\abs^{t'}$ is irreducible.
\item $u_\rho(a,b)\abs^{-\frac12}\times u_\rho(c,d)\abs^{\frac12}$ is reducible if and only if $(c,d)=(a,b)$.
\end{enumerate}
In particular, $u_\rho(a,b)\abs^{-\frac12}\times u_\rho(c,d)\abs^{s}$ is irreducible for $s\in\{0,-\tfrac12\}$.
\end{lemma}

\begin{proof}
Let $A=\frac{a+b}{2}-1$, $B=\frac{a-b}{2}$, $C=\frac{c+d}{2}-1$, $D=\frac{c-d}{2}$, and let $\delta:=t'-t$. Then the coordinate-wise difference of $q\big(u_\rho(c,d)\abs^{t'}\big)$ and $q\big(u_\rho(a,b)\abs^{t}\big)$ can be expressed as
\[
q=\big(A-C+\delta,\ D-B+\delta,\ B-D+\delta,\ C-A+\delta\big).
\]
Suppose that $u_\rho(c,d)\abs^{t'} \times u_\rho(a,b)\abs^{t}$ is reducible. By symmetry and Theorem \ref{thm:Tadic}, we may assume that the entries of $q$ are all positive integers.

For Part (a), we have $A-C+\delta \geq 1$ and $C-A+\delta \geq 1$. Therefore, $2 \delta \geq 2$, a contradiction. For Part (b), we have $\delta=1$. Then $A-C +\delta \geq 1$ and $C-A +\delta \geq 1$ imply that $A=C$. Similarly, we get $B=D$. Therefore, we must have $(a,b)=(c,d)$. This completes the proof of the lemma.
\end{proof}

\begin{lemma}\label{lem:half-commute}
Consider essentially generalized Speh representations $u= u_{\rho}(a,b)\lvert \cdot \rvert^{s}$ and $u'= u_{\rho}(c,d)\lvert \cdot \rvert^{t}$. Suppose that the centers of the highest segments of $u$ and of $u'$ are either $-1$ or $-\half{1}$. Namely, $ \{ \half{b-1}+s, \half{d-1}+t \} \subseteq \{-1, \half{-1}\}$. Then $u\times u'$ is irreducible. In particular $u\times u'=u'\times u$.
\end{lemma}

\begin{proof}
Write $q(u)=(z,w,x,y)$ and $q(u')=(z',w',x',y')$. Suppose the contrary that $u\times u'$ is reducible. Then by Theorem~\ref{thm:Tadic}, we may assume that $x\geq x'+1$ and $y\geq y'+1$. Then
\[  2(\half{b-1}+s)= x+y \geq x'+y'+2= 2(\half{d-1}+t) +2 \geq  0,  \]
a contradiction.
\end{proof}

\begin{lemma}\label{lem:Delta uj+}
    Let $\Delta=\Delta_{\rho}[x,y]$ be a segment with $e(\Delta)\leq -\half{1}$ and let $u= u_{\rho}(a,b)\abs^s$ be an essentially generalized Speh representation whose highest segment has center $-1$ or $-\half{1}$. Write $\Delta_u:=\Delta_{\rho}[\half{a-b}+s, -\half{a+b}+1+s]$ for the lowest segment of $u$, and write $u^{\flat}:= u_{\rho}(a,b-1)\abs^{s+\half{1}}$, so that $u \hookrightarrow \Delta_{u} \times u^{\flat}$.
\begin{enumerate}
    \item [(a)] Suppose that $e(\Delta_u) = e(\Delta)$. Then we have an injection
    \[ \Delta \times u \hookrightarrow  \Delta_{u}  \times \Delta \times u^{\flat}.\]
    \item [(b)] Suppose that $e(\Delta_u) < e(\Delta)$. Then we have
   \[ \Delta \times u = u \times \Delta \hookrightarrow  \Delta_{u}   \times u^{\flat}\times \Delta=\Delta_{u}   \times \Delta\times u^{\flat}.\]
\end{enumerate}
\end{lemma}
\begin{proof}
For Part (a), Theorem \ref{thm:Tadic} implies $\Delta_u \times \Delta$ is irreducible. Then by Lemma \ref{lem:socle}, we have 
\[ \Delta \times u \hookrightarrow \Delta \times \Delta_u \times u^{\flat}= \Delta_u \times  \Delta  \times u^{\flat}.\]

For Part (b), the injection in the middle again follows from Lemma \ref{lem:socle}, and it remains to prove the two equalities. Write $q(u)=(A,B,C,D)$, so that $q(u^{\flat})=(A+1, B+1, C,D)$; recall that $q(\Delta)=(y,x,y,x)$. By Theorem \ref{thm:Tadic}, if $\Delta\times u$ (resp.\ $\Delta\times u^{\flat}$) is reducible, then the exponents of $\Delta$ and $u$ lie in the same coset modulo $\Z$, and $\Delta$ is either strongly smaller or strongly larger than $u$ (resp.\ $u^{\flat}$).

First, suppose that $\Delta >_{\textrm{strong}} u$ or $\Delta >_{\textrm{strong}} u^{\flat}$. Then $y>C$ and $x>D$, and since $y-C,x-D\in\Z$, we get $x+y\geq C+D+2$. Our assumption on the highest segment of $u$ gives $\half{C+D}\geq -1$, hence $e(\Delta)=\half{x+y}\geq 0$, contradicting $e(\Delta)\leq-\half{1}$.

Next, if $\Delta <_{\textrm{strong}} u$, then $B>x$ and $A>y$, which imply that 
\[e(\Delta_{u})=\half{A+B} > \half{x+y}=e(\Delta),\]
a contradiction. Finally, if $\Delta <_{\textrm{strong}} u^{\flat}$, then $B+1>x$ and $A+1>y$, so $B \geq x$ and $A\geq y$, and hence
\[ e(\Delta_{u})=\half{A+B} \geq \half{x+y}=e(\Delta), \]
which again contradicts $e(\Delta_u)< e(\Delta)$. This completes the proof of the lemma.
\end{proof}
\subsection{Inductions from essentially Arthur type representations}

In the purely bad parity setting, M{\oe}glin showed that Arthur type representations are irreducibly induced from products of generalized Speh representations.

\begin{thm}[{\cite[\S 4.2]{Moe11}}]\label{thm:Moeglin}
Let $\{(\rho_i,a_i,b_i)\}_{i=1}^r$ be triples where $\rho_i$ is a self-dual unitary supercuspidal representation of $\GL_{d_i}(F)$ and $a_i,b_i\in\Z_{>0}$ satisfy $\alpha_{\rho_i}+\frac{a_i+b_i}{2}\notin\Z$. Then
\[
\bigtimes_{i=1}^r u_{\rho_i}(a_i,b_i)\rtimes\one_{G_0}
\]
is irreducible. Every representation of purely bad parity which is of Arthur type is of this form. In particular, representations of purely bad parity and of Arthur type are unitary.
\end{thm}

The irreducibility of inductions from essentially generalized Speh representations and a representation of Arthur type is controlled by the following result of Bo\v{s}njak--Stadler.

\begin{thm}[{\cite[Theorem~1.2]{BS25}}]\label{thm:BS}
Let $\tau_1,\dots,\tau_r$ be essentially Speh representations and let $\pi_A$ be a representation of $G$ of Arthur type. Then $\tau_1\times\cdots\times\tau_r\rtimes\pi_A$ is irreducible if and only if $\tau_i\times\tau_j$ and $\tau_i\times\tau_j^\vee$ are irreducible for all $i\neq j$ and $\tau_i\rtimes\pi_A$ is irreducible for all $i$.
\end{thm}

We also use the following special case of the irreducibility result of Atobe.

\begin{thm}[{\cite[Theorem~5.3]{Ato22b}}]\label{thm:Atobe}
Let $\rho$ be self-dual, $a,b\geq1$ and $s\in\R$. Then $u_\rho(a,b)\abs^s\rtimes\one$ is irreducible if $ \half{a+b} +s $ is not in $\alpha_{\rho} +\Z$.
\end{thm}

We end this section by showing the easier direction of Theorem \ref{thm:main}.

\begin{prop}\label{prop:A prime unitary}
    Suppose that $\pi$ is a representation of purely bad parity satisfying condition $(\mathrm{A}')$. Then $\pi$ is unitary. Moreover, in the second case of Definition \ref{def A prime}, the induction $\abs^{1/2}\rtimes\pi'$ is irreducible.
\end{prop}
\begin{proof}
    By Theorem \ref{thm:Moeglin}, we only need to consider the case that $G= \Sp_{2n}(F)$ and
    \[ \pi= \abs^{1/2} \times \bigtimes_{i \in I} u_{\rho_i}(a_i, b_i) \rtimes \pmb{1}_{\Sp_{0}}.\]
    Consider a family
    \[ \Pi_s:=\abs^{s} \times \bigtimes_{i \in I} u_{\rho_i}(a_i, b_i) \rtimes \pmb{1}_{\Sp_{0}} \]
    for $0 \leq s \leq \half{1}$. By Theorems \ref{thm:BS}, \ref{thm:Atobe} and Lemma \ref{lem:Speh-irr}, $\Pi_s$ is irreducible for all $ 0 < s \leq \half{1}$.  Moreover, $\Pi_{0}$ is also irreducible, since the good parity induction $\abs^0 \rtimes \pmb{1}_{\Sp_{0}}$ is irreducible. Since $\Pi_0$ is unitarily induced from a unitary representation, it is unitary. Hence, by \ref{U2}, $\{\Pi_{s}\}_{0 \leq s \leq \half{1}}$ is a continuous family of irreducible unitary representations. In particular, $\pi= \Pi_{\half{1}}$ is irreducible and unitary. This completes the proof of the proposition. 
\end{proof}

\section{Reduction to minimal counter examples}\label{sec:reduction}
We give a precise definition of \emph{minimal counter examples}, which were mentioned in the outline of the proof.

\begin{defn}\label{def:reduced}
 We say that a representation $\pi$ of purely bad parity is a \emph{minimal counter example} if $\pi$ is the Langlands subquotient of a parabolic induction of the form
\[
\St(\rho,a)\abs^{s}\times\St(\rho,a)\abs^{s}\times \bigtimes_{i\in I}u_{\rho_i}(a_i,b_i)\rtimes\pi_0,
\]
where
\begin{itemize}
\item $s<-\half{1}$ (and $s \in \half{1}\Z$ since $\pi$ is of purely bad parity);
\item $(\rho_i,a_i,b_i)\neq(\rho,a,-2s-1)$ for all $i\in I$;
\item if $G=\SO_{2n+1}$, then $\pi_0$ is the trivial representation of $\SO_1(F)$. If $G=\Sp_{2n}$, then $\pi_0$ is either the trivial representation of $\Sp_0(F)$ or the irreducible induced representation $\Ind_{\GL_1}^{\Sp_2}\abs^{1/2}=\abs^{1/2}\rtimes\one_{\Sp_0}$.
\end{itemize}
\end{defn}

In Notation \ref{not:UDelta}, we have $\widetilde{u}_\Delta=u_\rho(a,-2s-1)$, so the second condition says that $u_{\rho_i}(a_i,b_i)\not\cong \widetilde{u}_\Delta$ for all $i\in I$. We will see in Remark~\ref{rem:reduced-not-pA} that a minimal counter example does not satisfy condition $(\mathrm{A}')$. The goal of this section is to prove the following reduction.

\begin{prop}\label{prop:reduction}
Suppose that there exists a unitary representation of purely bad parity that does not satisfy condition $(\mathrm{A}')$. Then there exists a unitary representation of purely bad parity that is a minimal counter example.
\end{prop}

\subsection{A characterization of condition \texorpdfstring{$(\mathrm{A}')$}{(A')}}
In this subsection, we characterize condition $(\mathrm{A}')$ in terms of the multiplicities of segments in the $L$-data of $\pi$. 

\begin{defn}\label{def:mult}
Let $\pi=L(\Delta_1,\dots,\Delta_f;\ptemp)$ with $\ptemp=\Delta^0_1\times\cdots\times\Delta^0_k\rtimes\one$, where $e(\Delta^0_j)=0$. For a segment $\Delta$, define
\[
m_{\Delta}(\pi):=\begin{cases}
\#\{1\leq i\leq f\mid \Delta_i\cong\Delta\} & \text{if } e(\Delta)<0,\\
m_{\Delta^\vee}(\pi) & \text{if } e(\Delta)>0,\\
2\cdot\#\{1\leq j\leq k\mid \Delta^0_j\cong\Delta\} & \text{if } e(\Delta)=0.
\end{cases}
\]
\end{defn}

\begin{lemma}\label{lem:pA-type}
Let $\pi$ be a representation of $G$ of purely bad parity. Then $\pi$ satisfies condition $(\mathrm{A}')$ if and only if the following hold.
\begin{enumerate}[label=(\alph*)]
\item $m_{\Delta}(\pi)$ is even for all segments $\Delta$, except that $m_{\abs^{\pm 1/2}}(\pi)$ may be odd if $G=\Sp_{2n}(F)$.
\item $m_{\Delta}(\pi)\leq m_{\Delta\abs^1}(\pi)$ for all $\Delta$ with $e(\Delta)<0$.
\end{enumerate}
\end{lemma}

\begin{proof}
This follows from the definition and Theorem \ref{thm:Moeglin}. See the explanation for the characterization of Arthur type in the introduction. Note that if $\pi_0$ is non-trivial, i.e., $\pi_0= \abs^{1/2} \rtimes \pmb{1}_{\Sp_0}$, then $m_{\abs^{-1/2}}(\pi_0)=1$. We omit the rest of the details.
\end{proof}

\begin{remark}\label{rem:reduced-not-pA}
If $\pi$ is a minimal counter example, with $\Delta=\St(\rho,a)\abs^s$ as in Definition~\ref{def:reduced}, it is not hard to see that $m_{\Delta}(\pi)-m_{\Delta\abs^1}(\pi)\in\{1,2\}$. Moreover, it is $1$ if and only if $\Delta=\abs^{-3/2}$ and $\pi_0$ is non-trivial. In particular, minimal counter examples do not satisfy condition $(\mathrm{A}')$.
\end{remark}

\subsection{Parity constraint}
In this subsection, we prove the following proposition.

\begin{prop}\label{prop:parity}
Suppose that $\pi$ is a representation of $G$ of purely bad parity that is unitary. Then for every segment $\Delta$, $m_{\Delta}(\pi)$ is even, unless $G=\Sp_{2n}(F)$ and $\Delta\in\{\abs^{1/2}, \abs^{-1/2}\}$.
\end{prop}

\begin{proof}
Since $m_{\Delta}(\pi)=m_{\Delta^\vee}(\pi)$ and $m_{\Delta}(\pi)$ is even when $e(\Delta)=0$, it suffices to consider segments with $e(\Delta)<0$. Consider the set of segments
\[
\Dodd:=\{\Delta\mid e(\Delta)<0,\ m_{\Delta}(\pi)\text{ is odd}\},\qquad \Deven:=\{\Delta\mid e(\Delta)\le0,\ m_{\Delta}(\pi)\in2\Z_{>0}\}.
\]
 Define $\pi'$ to be the Langlands subquotient of
\[
\bigtimes_{\Delta\in\Dodd}\big(\widetilde{u}_\Delta\big)^{\frac{m_{\Delta}(\pi)+1}{2}}\times\bigtimes_{\Delta\in\Deven}\big(\widetilde{u}_\Delta\big)^{\frac{m_{\Delta}(\pi)}{2}}\rtimes\pi .
\]
If $\pi$ is unitary, then the above induced representation is unitarily induced from a unitary representation, hence semisimple with all subquotients unitary by \ref{U1}. Therefore, the Langlands subquotient $\pi'$ is unitary.

On the other hand, comparing the $L$-parameter, $\pi'$ is also the Langlands subquotient of the induction
\[
\Pi':=\bigtimes_{\Delta\in\Dodd}L(\Delta\times \widetilde{u}_\Delta)\times\bigtimes_{\Delta\in\Dodd}u_\Delta^{\frac{m_{\Delta}(\pi)-1}{2}}\times\bigtimes_{\Delta\in\Deven}u_\Delta^{\frac{m_{\Delta}(\pi)}{2}}\rtimes\one .
\]
Write $\Delta=\St(\rho_\Delta,a_\Delta)\abs^{\frac{1-b_\Delta}{2}}$. Then $L(\Delta\times \widetilde{u}_\Delta)=u_{\rho_\Delta}(a_\Delta,b_\Delta-1)\abs^{-\frac12}$ by Lemma \ref{lem:socle}. We claim that $\Pi'$ is irreducible, and hence $\pi'=\Pi'$. More generally, consider the representations
\[
\Pi'_s:=\bigtimes_{\Delta\in\Dodd}u_{\rho_\Delta}(a_\Delta,b_\Delta-1)\abs^{-\frac12+s}\rtimes\pi_A,\qquad \pi_A:=\bigtimes_{\Delta\in\Dodd\cup\Deven}u_\Delta^{\lfloor\frac{m_{\Delta}(\pi)}{2}\rfloor}\rtimes\one.
\]
We claim that $\Pi_s'$ are irreducible for all $0\leq s<\frac12$ (note that $\Pi'_0=\Pi'$). 

Indeed, by Theorem \ref{thm:BS}, it suffices to show that for any two essentially generalized Speh representations $u=u_{\rho}(a,b)\abs^t$ and $u'= u_{\rho'}(c,d)\abs^{t'}$ corresponding to different indices in the product of $\Pi'_{s}$, the inductions $u \times u', u^{\vee} \times u'$ and $u \rtimes \pmb{1}$ are all irreducible. Since $|t\pm t'| \in \{ 0, \half{1}-s, 1-2s \}$, Lemma \ref{lem:Speh-irr} implies that $u \times u'$ is always irreducible and $u^{\vee} \times u'$ could be reducible only when $s=0$ and $u=u'$ are both factors of the product over $\Dodd$. However, this is impossible since the product over $\Dodd$ is multiplicity free. Finally, the induction $u \rtimes \pmb{1}$ is always irreducible by Theorem \ref{thm:Atobe}. This verifies the claim.

Now $\Pi'_s$ is irreducible for $0\leq s<\frac12$, and $\Pi'_0=\Pi'=\pi'$. Since $\pi'$ is unitary, $\Pi'_s$ is unitary for all $0\leq  s<\frac12$ by \ref{U2}.  By the classification of unitary complementary Arthur representations in \cite[Theorem~1.3]{HJLLZ25} or \cite[Theorem 4.1]{AM25}, the induction $u_{\rho_\Delta}(a_\Delta,b_\Delta-1)\rtimes\pi_A$ must be irreducible for all $\Delta \in \Dodd$. Applying Theorem \ref{thm:BS} again, we see that the (good parity) induction
\[u_{\rho_\Delta}(a_\Delta,b_\Delta-1)\rtimes\one\]
must be irreducible for all $\Delta \in \Dodd$. By Lemma \ref{lem:good parity induction} below, this happens only when $G=\Sp_{2n}(F)$, $\rho_\Delta=\pmb{1}$, $a_{\Delta}=1$ and $b_{\Delta}=2$, so $\Delta=\abs^{-1/2}$. We conclude that $\Dodd$ must be empty if $G_n= \SO_{2n+1}(F)$, and $\Dodd \subseteq \{\abs^{-1/2}\}$ if $G_n= \Sp_{2n}(F)$. This completes the proof of the proposition.
\end{proof}

Finally, we prove the missing piece of the above proof uses extended multi-segments, for which we refer to \cite{Ato22a} and \cite[Section~4.3]{Ato22b}.

\begin{lemma}\label{lem:good parity induction}
Suppose that $\rho$ is self-dual and $a,b \in \Z_{\geq 1}$ such that
$\half{a+b} \in \alpha_{\rho}+\Z$. Then the (good parity) induction
    \[ u_{\rho}(a,b) \rtimes \pmb{1}\]
    is irreducible if and only if $G= \Sp_{2n}(F)$, $\rho$ is trivial and $a=b=1$.
\end{lemma}
\begin{proof}
    By \cite[Proposition~4.2]{Ato22b}, the induction $u_{\rho}(a,b) \rtimes \pmb{1}$ is multiplicity-free, and its irreducible subquotients are exactly the members of the local Arthur packet $\Pi_{\psi}(G)$, where
    \[ \psi= \begin{cases}
         (\rho \otimes S_{a} \otimes S_b)^{\oplus 2} & \text{ if }G= \SO_{2n+1}(F),\\
         (\rho \otimes S_{a} \otimes S_b)^{\oplus 2}+ 1 \otimes S_1 \otimes S_1 & \text{ if }G= \Sp_{2n}(F).
    \end{cases}\]
    Since the Aubert--Zelevinsky involution commutes with parabolic induction up to semisimplification and sends $u_\rho(a,b)$ to $u_\rho(b,a)$, we may assume that $a \geq b$. Set $A:= \half{a+b}-1$ and $B:= \half{a-b}$. By \cite[Theorem~4.4]{Ato22b}, $u_\rho(a,b)\rtimes\pmb{1}$ decomposes as the direct sum of the non-zero $\pi(\mathcal{E})$, where $\mathcal{E}$ is obtained from the extended multi-segment of $\pmb{1}$ by adding $([A,B]_\rho,l,\eta)$ and $([A,B]_\rho,l,(-1)^{A-B}\eta)$ with $0\leq l\leq \half{b}$ and $\eta\in\{\pm1\}$. Hence it suffices to exhibit two such $\mathcal{E}$ with $\pi(\mathcal{E})\neq 0$.
    
If $G=\SO_{2n+1}(F)$, we take the extended multi-segments $\mathcal{E}_+$ and $\mathcal{E}_-$, where
\[ \mathcal{E}_\pm= \{([A,B]_{\rho},0,\pm1), ([A,B]_{\rho},0, \pm (-1)^{A-B})\}.\]
If $G=\Sp_{2n}(F)$ and $(\rho,a,b) \neq (\pmb{1},1,1)$, we take the following extended multi-segments.
\begin{itemize}
    \item If $\rho \neq \pmb{1}$, or $\rho= \pmb{1}$ and $a>b$, then consider $\{\mathcal{E}_+, \mathcal{E}_-\}$, where
    \[ \mathcal{E}_\pm= \{([A,B]_{\rho},0,\pm1), ([A,B]_{\rho},0, \pm(-1)^{A-B})\} \cup \{([0,0]_{\pmb{1}},0,1)\}.\]
    \item If $\rho =\pmb{1}$ and $a=b$, then $b\geq 2$ since $(a,b) \neq (1,1)$. In this case we consider $\mathcal{E}_+$ and $\mathcal{E}_1$, where
    \begin{align*}
         \mathcal{E}_+&= \{([A,B]_{\rho},0,1), ([A,B]_{\rho},0, (-1)^{A-B})\} \cup \{([0,0]_{\pmb{1}},0,1)\},\\
         \mathcal{E}_1&= \{([A,B]_{\rho},1,1), ([A,B]_{\rho},1, (-1)^{A-B})\} \cup \{([0,0]_{\pmb{1}},0,1)\}.
    \end{align*}
\end{itemize}
By the non-vanishing criterion \cite[Theorems~3.7 and 4.4]{Ato22a}, one checks that the representations attached to the above extended multi-segments are non-zero. Hence $u_\rho(a,b)\rtimes\pmb{1}$ is reducible in these cases. Conversely, if $G=\Sp_{2n}(F)$ and $(\rho,a,b)=(\pmb{1},1,1)$, then $\abs^0\rtimes\one_{\Sp_0}$ is irreducible. This completes the proof of the lemma.
\end{proof}

\subsection{Proof of \texorpdfstring{Proposition~\ref{prop:reduction}}{Proposition 4.2}}

Suppose that $\pi$ is a unitary representation of purely bad parity which does not satisfy condition $(\mathrm{A}')$. We use the notation $\Dodd$, $\Deven$ of the proof of Proposition~\ref{prop:parity}. By Proposition~\ref{prop:parity}, $\Dodd$ is either empty or equal to $\{\abs^{-1/2}\}$ (the latter only if $G=\Sp_{2n}(F)$). Hence condition (a) of Lemma~\ref{lem:pA-type} holds, so condition (b) must fail. Since (b) is automatic when $e(\Delta)=-\frac12$, the set
\[
\mathcal D:=\{\Delta\in\Deven\mid e(\Delta)<-\tfrac12,\ m_{\Delta}(\pi)>m_{\Delta\abs^1}(\pi)\}
\]
is nonempty. We fix $\Delta_0\in\mathcal D$ such that $e(\Delta_0)=\max_{\Delta\in\mathcal D}e(\Delta)$, and set
\[
k_0:=\frac{m_{\Delta_0}(\pi)}{2}-\Big\lfloor\frac{m_{\Delta_0\abs^1}(\pi)}{2}\Big\rfloor-1\ \geq 0 .
\]
Here the floor function is needed only when $\Delta_0= \abs^{-3/2}$. Let $\pi_0:=\bigtimes_{\Delta'\in\Dodd}\Delta'\rtimes\one$, and let $\pi'$ be the Langlands subquotient of
\[
\big(\widetilde{u}_{\Delta_0}\big)^{k_0}\times\bigtimes_{\Delta\in\Deven\setminus\{\Delta_0\}}\big(\widetilde{u}_\Delta\big)^{\frac{m_{\Delta}(\pi)}{2}}\rtimes\pi .
\]
As in the proof of Proposition~\ref{prop:parity}, $\pi'$ is unitary by \ref{U1}. Comparing the $L$-parameter by Lemma~\ref{lem:socle}, $\pi'$ is also the Langlands subquotient of
\[
\Delta_0^2\times u_{\Delta_0}^{\,k_0+\lfloor m_{\Delta_0|\cdot|^1}(\pi)/2\rfloor}\times\bigtimes_{\Delta\in\Deven\setminus\{\Delta_0,\Delta_0\abs^1\}}u_\Delta^{\frac{m_{\Delta}(\pi)}{2}}\rtimes\pi_0.
\]
Then it is clear from this expression that $\pi'$ is a minimal counter example (with respect to $\Delta_0$). This completes the proof of the proposition.
 \qed

\section{Minimal counter examples are not unitary}\label{sec:main}

The goal of this section is to prove the following theorem.

\begin{thm}\label{thm:reduced}
Suppose that $\pi$ is a minimal counter example; see Definition \ref{def:reduced}. Namely $\pi$ is the Langlands subquotient of
\[
\Delta\times\Delta\times\bigtimes_{i\in I}u_{\rho_i}(a_i,b_i)\rtimes\pi_0,
\]
where $\Delta=\seg{\rho}{x}{y}$ with $e(\Delta)<-\frac12$ and $u_{\rho_i}(a_i,b_i)\not\cong \widetilde{u}_\Delta$ for all $i\in I$. Then $\pi$ is not unitary.
\end{thm}

The main idea is to show that the irreducible representation satisfying condition $(\mathrm{A}')$
\begin{equation}\label{eq tau}
    \tau:=u_\Delta\times\bigtimes_{i\in I}u_{\rho_i}(a_i,b_i)\rtimes\pi_0,
\end{equation}
which is the Langlands subquotient of $\widetilde{u}_\Delta\rtimes\pi$, cannot be a subrepresentation of $\widetilde{u}_\Delta\rtimes\pi$. We will design a composition of derivatives which is nonzero on every irreducible subrepresentation $\sigma$ of $\widetilde{u}_\Delta\rtimes\pi$, but vanishes on $\tau$. The main idea for designing these derivatives is that the lowest segment of $\widetilde{u}_\Delta$ is ``locked" in $u_{\Delta}$ for $\tau$. For example, we have $L_{\rho\abs^{x+1}} (\widetilde{u}_\Delta) \neq 0$ but $L_{\rho\abs^{x+1}} (u_{\Delta}) =0$. Thus, the derivatives of the lowest segment of $\widetilde{u}_\Delta$ will contribute to that of a subrepresentation of $\widetilde{u}_\Delta\rtimes\pi$, but not to that of $\tau$.

\subsection{Notation}
Now we fix the notation for the rest of the section. Thus, we fix a minimal counter example from Theorem \ref{thm:reduced}. For $i\in I$, set
\begin{align*}
u_i&:=u_{\rho_i}(a_i,b_i),\\
u_i^{-}&:=u_{\rho_i}\big(a_i,\lfloor\tfrac{b_i}{2}\rfloor\big)\abs^{\frac{\lfloor b_i/2\rfloor-b_i}{2}},\\
u_i^{+}&:=u_{\rho_i}\big(a_i,\lfloor\tfrac{b_i}{2}\rfloor\big)\abs^{-\frac{\lfloor b_i/2\rfloor-b_i}{2}},\\
u_i^{0}&:=\begin{cases}\text{the trivial representation of }\GL_0(F)&\text{if } b_i\text{ is even},\\ u_{\rho_i}(a_i,1)=\St(\rho_i,a_i)&\text{otherwise}.\end{cases}
\end{align*}
Thus $u_i^-$ consists of the segments of $u_i$ with negative center, $u_i^+$ of those with positive center, and
\[
u_i\hookrightarrow u_i^{-}\times u_i^{0}\times u_i^{+}.
\]
Here if $b_i=1$, then $u_i^{\pm}$ are trivial representation of $\GL_0(F)$. Set $\ptemp:=\bigtimes_{i\in I}u_i^0\rtimes\one$, which is irreducible and tempered. Also, set
\[ \pi_{0,temp}:= \bigtimes_{i\in I}u_i^0\rtimes \pi_0,\]
which is an irreducible induction. We also set
\[
\Delta_i:=\seg{\rho_i}{x_i}{y_i},\qquad x_i:=\frac{a_i-b_i}{2},\quad y_i:=1-\frac{a_i+b_i}{2},
\]
the lowest segment of $u_i$, and
\[
\widetilde u_i^{-}:=\begin{cases}u_{\rho_i}\big(a_i,\lfloor\tfrac{b_i-2}{2}\rfloor\big)\abs^{\frac{\lfloor (b_i-2)/2\rfloor-(b_i-2)}{2}}&\text{if } b_i\geq2,\\ \text{the trivial representation of }\GL_0(F)&\text{otherwise},\end{cases}
\]
so that $\widetilde u_i^-=u_{\rho_i}(a_i,b_i-2)^-$ and, for $b_i\geq 2$,
\[
u_i^{-}\hookrightarrow\Delta_i\times\widetilde u_i^{-}.
\]
Note that the highest segment of $u_i^-$ (if $b_i\geq2$) has center $-1$ or $-\frac12$, and its lowest segment is $\Delta_i$, with center $\frac{1-b_i}{2}$. 

Finally, we decompose $I=I_1 \sqcup I_2$, where 
\[
I_2:=\Big\{i\in I\ \Big|\ \rho_i\cong\rho,\ x_i=x+1,\ y_i>y\Big\}=\{i\in I\mid u_i^->_{\mathrm{strong}}\Delta,\ x_i=x+1\},
\]
and $I_1:=I\setminus I_2$.

\subsection{Embeddings into standard modules}
In this subsection, we prove a technical proposition that allows us to construct various embeddings.

\begin{prop}\label{prop:embed}
    Suppose that $\Delta^1,\dots,\Delta^f$ are segments with $e(\Delta^1)\leq\cdots\leq e(\Delta^f)\leq -\half{1}$ and the Langlands subquotient $\pi$ of an induction
    \[ \Delta^1\times\cdots\times\Delta^f\times\bigtimes_{i}u_i\rtimes\pi_0\]
    is of purely bad parity, where $\pi_0=\pmb{1}_{\SO_{1}}$ if $G_n= \SO_{2n+1}(F)$, and $\pi_0 \in \{ \pmb{1}_{\Sp_{0}}, \abs^{1/2}\rtimes \pmb{1}_{\Sp_0}\}$ if $G_n=\Sp_{2n}(F)$. Set $\pi_{0,temp}:= \bigtimes_{i \in I}u_i^0 \rtimes \pi_0$.  Then there is an embedding
\[
\Delta^1\times\cdots\times\Delta^f\times\bigtimes_{i\in I}(u_i^-)^2\rtimes\pi_{0,temp}\hookrightarrow M(\pi),
\]
where the product $\bigtimes_{i\in I}(u_i^-)^2$ can be taken in an arbitrary order. In particular, $\pi$ is the unique irreducible subrepresentation of the left hand side.
\end{prop}
\begin{proof}
    The order of the product $\bigtimes_{i\in I}(u_i^-)^2$ can be taken arbitrarily by Lemma \ref{lem:half-commute}, since the center of the highest segment of each $u_i^-$ is either $-1$ or $-\half{1}$. Write 
    \[ \pi= L(\Delta_1',\ldots, \Delta_{r}'; \pi_{temp}),\]
    where recall that we require $e(\Delta_1') \leq \cdots \leq e(\Delta_r') <0$. We apply induction on $r$.
    Suppose that $r=0$. Then we have  $\pi=\pi_{temp}=M(\pi)$ and the embedding trivially holds.  From now on we assume that $r \geq 1$.

 We first deal with the easier case that $\Delta_1'= \Delta^k$ for some $1 \leq k \leq f$ (so $e(\Delta^1)=\cdots=e(\Delta^k)$). By relabeling the $\Delta^i$ we may assume that $k=1$.
 Then set 
 \[\pi^-= L(\Delta_2',\ldots, \Delta_{r}'; \pi_{temp}).\]
The induction hypothesis for $\pi^-$ gives 
\[\Delta^1 \times \left(\Delta^2\times\cdots\times\Delta^f\times\bigtimes_{i\in I}(u_i^-)^2\rtimes\pi_{0,temp} \right)\hookrightarrow \Delta^1 \rtimes M(\pi^-)= M(\pi). \]
From now on, we assume that $\Delta_1'\neq\Delta^i$ for all $i$, so that $e(\Delta_1')\leq  e(\Delta^1)$.

Next, we deal with the case that $\Delta_1'=\Delta_j$ for some $j \in I$. In this case, we have $m_{\Delta_1'}(\pi)\geq 2$, so we may assume that $\Delta_2'=\Delta_1'$ by relabeling if necessary.
Set 
\[ k:= \max(\{0\}\sqcup \{1 \leq i \leq f\ | \ e(\Delta^i)=e(\Delta_1') \} ).\]
Then, applying Lemma \ref{lem:Delta uj+} repeatedly (with $u=u_j^-$, in which case $u^{\flat}=\widetilde{u}_j^-$), we obtain injections
\begin{align*}
\bigtimes_{1 \leq i \leq k} \Delta^i \times \bigtimes_{k+1 \leq i \leq f} \Delta^i  \times (u_j^-)^2 &\hookrightarrow \bigtimes_{1 \leq i \leq k} \Delta^i  \times (u_j^-)^2 \times \bigtimes_{k+1 \leq i \leq f} \Delta^i  \\
& \hookrightarrow \bigtimes_{1 \leq i \leq k} \Delta^i  \times (\Delta_j)^2 \times  (\widetilde{u}_j^-)^2 \times \bigtimes_{k+1 \leq i \leq f} \Delta^i  \\
&=(\Delta_j)^2 \times \bigtimes_{1 \leq i \leq k} \Delta^i  \times    \bigtimes_{k+1 \leq i \leq f} \Delta^i  \times (\widetilde{u}_j^-)^2.
\end{align*}
Now put 
\[ \pi^-:= L(\Delta_3',\ldots, \Delta_{r}'; \pi_{temp}).\]
The above computation and the induction hypothesis for $\pi^-$ give 
 \begin{align*}
&\Delta^1\times\cdots\times\Delta^f\times\bigtimes_{i\in I}(u_i^-)^2\rtimes\pi_{0,temp}\\
&\hookrightarrow (\Delta_j)^{2}\times \left(\Delta^1\times\cdots\times\Delta^f \times (\widetilde{u}_j^-)^2\times\bigtimes_{i\in I \setminus\{j\}}(u_i^-)^2\rtimes\pi_{0,temp}\right)\\
&\hookrightarrow (\Delta_j)^{2} \times M(\pi^-)\\
&= M(\pi).
 \end{align*}

Finally, we deal with the remaining case: $\Delta_1' \neq \Delta^i$ for all $i$. In this case, the segment $\Delta_1'$ must come from $\pi_0$, so $\Delta_1'= \abs^{-1/2}$ and $m_{\Delta_1'}(\pi)=1$. If $r>1$, then we can swap $\Delta_1'$ with $\Delta_2'$. Since $\Delta_2'$ must equal to $\Delta^i$ for some $1 \leq i \leq f$ or $\Delta_j$ for some $j \in I$, we reduce the situation to the previous cases. If $r=1$, then $f=0$, and the induction is irreducible by Proposition \ref{prop:A prime unitary}. This completes the proof of the proposition.
\end{proof}

\subsection{The case \texorpdfstring{$I_2=\emptyset$}{I2 empty}}

In this subsection we prove the special case that $I_2$ is empty. We further divide this case into three subcases:
\begin{enumerate}
    \item [(i)] if $G_n=\Sp_{2n}(F)$, $\pi_0=\abs^{1/2} \rtimes \one$ and $\rho \cong \pmb{1}_{\GL_1}$, then $x+1 \not\in \{\pm \half{1}\}$ (Proposition~\ref{prop:empty});
    \item [(ii)] $G_n=\Sp_{2n}(F)$, $\pi_0=\abs^{1/2} \rtimes \one$, $\rho \cong \pmb{1}_{\GL_1}$, and  $x+1 \in \{\pm \half{1}\}$ but $\Delta \neq \abs^{-3/2}$ (Proposition~\ref{prop:empty-half});
    \item [(iii)] $G_n=\Sp_{2n}(F)$, $\pi_0=\abs^{1/2} \rtimes \one$, and $\Delta = \abs^{-3/2}$ (Proposition~\ref{prop:empty-32}).
\end{enumerate}

\begin{prop}\label{prop:empty}
Suppose that $\pi$ is as in Theorem~\ref{thm:reduced} with $I_2=\emptyset$. If $G=\Sp_{2n}(F)$, $\pi_0=\abs^{1/2}\rtimes\one_{\Sp_0}$ and $\rho\cong\one_{\GL_1}$, assume moreover that $x+1\neq\pm\frac12$. Let
\[
s:=|\{i\in I\mid \rho_i\cong\rho,\ x_i=x+1\}|.
\]
\begin{enumerate}
    \item [(a)]If $\sigma\hookrightarrow \widetilde{u}_\Delta\rtimes\pi$, then $D^{(2s+1)}_{\rho\abs^{x+1}}(\sigma)\neq0$. 
    \item [(b)] Let $\tau$ be defined in \eqref{eq tau}. Then $D^{(2s+1)}_{\rho\abs^{x+1}}(\tau)=0$.
\end{enumerate}
In particular, $\pi$ is not unitary.
\end{prop}
\begin{proof}
 Note that since $e(\Delta)<-1/2$ and $x+y \in \half{1}\Z$, we have $\half{x+y} \leq -1$ and hence $y\leq -x-2$.  
Set 
\[ I^{=x+1}=\{i \in I \ | \ x_i=x+1, \text{and }\rho_i \cong \rho\}, \]
 and decompose $I= I^{=x+1} \sqcup I^{\neq x+1}.$ Suppose that $i \in I^{=x+1}$. Then since $I_2$ is empty, we have
\[ y_i \leq y \leq-x-2.\]
In particular, we observe that for $i \in I^{=x+1}$,
\begin{itemize}
    \item the generalized Speh representation $u_i$ is not tempered, so $u_i^-$ is non-trivial, and
    \item $\Delta_i \times \Delta= \Delta \times \Delta_i$, since $\Delta\subseteq\Delta_i$.
\end{itemize}

Set 
\[ \widetilde{\pi}:= \bigtimes_{i \in I^{=x+1}} \widetilde{u}_i \times \bigtimes_{i \in I^{\neq x+1}} u_i \rtimes \pi_{0}. \]
Applying Proposition \ref{prop:embed} to $\widetilde{\pi}$ and $\pi$ respectively, we see that
\begin{align*}
    &\left(\bigtimes_{i \in I^{=x+1}} (\Delta_i)^{2} \times \Delta^2 \right) \rtimes \widetilde{\pi} \\
    \hookrightarrow & \left(\bigtimes_{i \in I^{=x+1}} (\Delta_i)^{2} \times \Delta^2 \right) \times \bigtimes_{i \in I^{=x+1}} (\widetilde{u}_i^-)^2 \times \bigtimes_{i \in I^{\neq x+1}} (u_i^-)^2 \rtimes \pi_{0,temp}\\
    \hookrightarrow& M(\pi).
\end{align*}
As a consequence, we have an injection
\[ \pi \hookrightarrow \left(\bigtimes_{i \in I^{=x+1}} (\Delta_i)^{2} \times \Delta^2 \right)  \rtimes \widetilde{\pi}.\]
Now we focus on the $\GL$ part. By Frobenius reciprocity, Leibniz rule and \eqref{eq der St} for the highest derivative of the irreducible induction $\bigtimes_{i \in I^{=x+1}} (\Delta_i)^{2}$, we have
\begin{align*}
   \left(\bigtimes_{i \in I^{=x+1}} (\Delta_i)^{2} \times \Delta^2 \right)  
    & \hookrightarrow  (\rho\abs^{x+1})^{2s} \times \left (\bigtimes_{i \in I^{=x+1}} (\Delta_{\rho}[x_i-1, y_i])^{2} \times \Delta^2 \right)  \\
    &=  (\rho\abs^{x+1})^{2s}  \times \Pi,
\end{align*}
where $\Pi$ is the remaining product. As a consequence, if $\sigma$ embeds into $\widetilde{u}_\Delta \rtimes \pi$, then 
\[ \sigma \hookrightarrow \widetilde{u}_\Delta \rtimes \pi \hookrightarrow \widetilde{u}_\Delta  \times (\rho\abs^{x+1})^{2s}  \times \Pi \rtimes \widetilde{\pi}.  \]
Again focus on the $\GL$ part. Since $L_{\rho\abs^{x+1}}(\widetilde{u}_\Delta)$ is a nonzero highest derivative, we have 
\[ \widetilde{u}_\Delta  \times (\rho\abs^{x+1})^{2s}=(\rho\abs^{x+1})^{2s} \times  \widetilde{u}_\Delta  \hookrightarrow  (\rho\abs^{x+1})^{2s+1} \times  L_{\rho\abs^{x+1}}(\widetilde{u}_\Delta).\]
Now Part (a) follows from Frobenius reciprocity.

For Part (b), the Leibniz rule gives
\[ D_{\rho\abs^{x+1}}^{(d)}(\tau)= \sum c_{(\underline{l},\underline{r},t )}  L_{\rho\abs^{x+1}}^{(l_\Delta)} \circ R_{\rho\abs^{-x-1}}^{(r_\Delta)} (u_{\Delta}) \times  \bigtimes_{i \in I} L_{\rho\abs^{x+1}}^{(l_i)} \circ R_{\rho\abs^{-x-1}}^{(r_i)} (u_i) \rtimes D_{\rho\abs^{x+1}}^{(t)}(\pi_0), \]
where the sum is taken over all tuples of non-negative integers \[(\underline{l}=(l_{\Delta}, \{l_i\}_{i \in I}),\ \underline{r}= (r_{\Delta}, \{r_i\}_{i \in I}),\ t )\] such that
\[ l_{\Delta}+r_{\Delta}+ \sum_{i} l_i+r_i +t=d,  \]
and each $c_{(\underline{l},\underline{r},t )}$ is a positive constant. However, Lemma \ref{lem:consecutive der Speh} implies that $u_{\Delta}$ is left $\rho\abs^{x+1}$-reduced and right $\rho\abs^{-x-1}$-reduced, and so is $u_i$ unless $\rho_i \cong \rho$ and $x_i=x+1$. Moreover, if  $\rho_i \cong \rho$ and $x_i=x+1$, then $L_{\rho\abs^{x+1}}^{(l_i)} \circ R_{\rho\abs^{-x-1}}^{(r_i)} (u_i) = 0$ unless $l_i \leq 1$ and $r_i \leq 1$. Finally, the representation $\pi_0$ is $\rho\abs^{x+1}$-reduced thanks to the assumption that $|x+1| \neq 1/2$. We conclude that if $D_{\rho\abs^{x+1}}^{(d)}(\tau)\neq 0$, then
\[ d \leq 2 |\{i \in I \ | \ \rho_i \cong \rho,\ x_i=x+1 \}|= 2s.\]
This completes the proof of Part (b). 
\end{proof}

If $|x+1|=1/2$, $\rho\cong\one_{\GL_1}$ and $\pi_0$ is non-trivial, the last step of the above proof shows that
\begin{align}\label{eq I2 nonempty (ii) tau}
    D_{\rho\abs^{x+1}}^{(2s+1)}(\tau)= u_{\Delta} \times \bigtimes_{i \in I^{\neq x+1}} u_i \times \bigtimes_{i \in I^{=x+1}} L_{\rho\abs^{x+1}} \circ R_{\rho\abs^{-x-1}}(u_i) \rtimes D_{\rho\abs^{x+1}}(\pi_0) 
\end{align}
is the highest derivative. In these cases, we further take derivatives at $x$, or at $x+2$ when $\Delta=\abs^{-3/2}$. We first consider the case that $\Delta \neq \abs^{-3/2}$.

\begin{prop}\label{prop:empty-half}
Suppose that $\pi$ is as in Theorem~\ref{thm:reduced} with $I_2=\emptyset$, that $G=\Sp_{2n}(F)$, $\pi_0=\abs^{1/2}\rtimes\one_{\Sp_0}$, $\rho\cong\one_{\GL_1}$ and $x+1\in\{\pm\frac12\}$, and that $\Delta\neq\abs^{-3/2}$. Let $s$ be as in Proposition~\ref{prop:empty} and
\[
t:=|\{i\in I\mid\rho_i\cong\rho,\ x_i=x+1,\ y_i\leq x\}|+|\{i\in I\mid\rho_i\cong\rho,\ x_i=x\}|+1 .
\]
\begin{enumerate}
 \item [(a)]If $\sigma\hookrightarrow \widetilde{u}_\Delta\rtimes\pi$, then $D^{(2t+1)}_{\rho\abs^{x}}\circ D^{(2s+1)}_{\rho\abs^{x+1}}(\sigma)\neq0$. 
    \item [(b)] Let $\tau$ be defined in \eqref{eq tau}. Then $D^{(2t+1)}_{\rho\abs^{x}}\circ D^{(2s+1)}_{\rho\abs^{x+1}}(\tau)=0$.
\end{enumerate}
In particular, $\pi$ is not unitary.
\end{prop}
\begin{proof}
    We continue the computation in Proposition \ref{prop:empty}. For Part (a), recall that we have 
    \begin{align*}
        \pi &\hookrightarrow (\rho\abs^{x+1})^{2s} \times \left (\bigtimes_{i \in I^{=x+1}} (\Delta_{\rho}[x, y_i])^{2} \times \Delta_{\rho}[x,y]^2 \right) \rtimes \widetilde{\pi},
    \end{align*}
where 
\[ \widetilde{\pi}:= \bigtimes_{i \in I^{=x+1}} \widetilde{u}_i \times \bigtimes_{i \in I^{\neq x+1}} u_i \rtimes \pi_{0}. \]
The product $\bigtimes_{i \in I^{=x+1}} (\Delta_{\rho}[x, y_i])^{2} \times \Delta_{\rho}[x,y]^2$ already contributes 
    \[2t_1+2:=2 |\{i\in I\mid\rho_i\cong\rho,\ x_i=x+1,\ y_i\leq x\}| +2\]
    to the order of the highest derivatives at $\rho\abs^{x}$. We will get the rest from $\widetilde{\pi}$ and $\widetilde{u}_\Delta$. 
    
     Since $\widetilde{\pi}$ is an irreducible induction, we may compute the order of the highest derivative of $\widetilde{\pi}$ with respect to $\rho\abs^{1/2}$ by the Leibniz rule and Lemma \ref{lem:consecutive der Speh}, similar to the proof of Part (b) of Proposition \ref{prop:empty}. We have:
    \begin{itemize}
        \item each term $\widetilde{u}_i$ in the product $\bigtimes_{i \in I^{=x+1}} \widetilde{u}_i $ does not contribute, since the only possible non-zero first left and right derivatives of $\widetilde{u}_i$ are $L_{\rho\abs^{x+2}}$ and $R_{\rho\abs^{-x-2}}$;
    \item each term $u_i$ in the product $\bigtimes_{i \in I^{\neq x+1}} u_i $ will contribute $2$ to the order of the highest derivative if $x_i=x$;
        \item the term $\pi_0$ will contribute $1$ to the order of the highest derivatives if $x=-1/2$, and zero otherwise.
    \end{itemize}
    Therefore, $D_{\rho\abs^{x}}^{(2t_2)} (\widetilde{\pi}) \neq 0$, where 
    \[ t_2=|\{i\in I\mid\rho_i\cong\rho,\ x_i=x\}|= t-t_1-1.\]
   As a consequence, there exists an irreducible representation $\widetilde{\pi}'$ such that 
   \[ \widetilde{\pi} \hookrightarrow (\rho\abs^{x})^{2 t_2} \rtimes \widetilde{\pi}',\]
   and
   \begin{align*}
       \pi &\hookrightarrow  (\rho\abs^{x+1})^{2s} \times \left (\bigtimes_{i \in I^{=x+1}} (\Delta_{\rho}[x, y_i])^{2} \times \Delta_{\rho}[x,y]^2 \right) \times (\rho\abs^{x})^{2t_2} \rtimes \widetilde{\pi}'\\
       &= (\rho\abs^{x+1})^{2s} \times (\rho\abs^{x})^{2 t_2} \times \left (\bigtimes_{i \in I^{=x+1}} (\Delta_{\rho}[x, y_i])^{2} \times \Delta_{\rho}[x,y]^2 \right)  \rtimes \widetilde{\pi}'\\
       &\hookrightarrow (\rho\abs^{x+1})^{2s} \times (\rho\abs^{x})^{2t_2 + 2t_1 +2} \times \left (\bigtimes_{i \in I^{=x+1}} (\Delta_{\rho}[x-1, y_i])^{2} \times \Delta_{\rho}[x-1,y]^2 \right)  \rtimes \widetilde{\pi}'\\
       &= (\rho\abs^{x+1})^{2s} \times (\rho\abs^{x})^{2t}  \rtimes \Pi,
   \end{align*}
where $\Pi$ is the remaining part. Thus, if $\sigma \hookrightarrow \widetilde{u}_\Delta \rtimes \pi$, then
\begin{align} \label{eq I2 empty (ii)}
   \sigma &\hookrightarrow  \widetilde{u}_\Delta \times (\rho\abs^{x+1})^{2s} \times (\rho\abs^{x})^{2t}  \rtimes \Pi.
\end{align}

Finally, observe that if $x+1=\half{1}$, then $y\leq x -1$ by our requirement that $e(\Delta)=\half{x+y} <-\half{1}$. If $x+1=-\half{1}$, then our assumption that $\Delta\neq \abs^{-3/2}$ again implies that $y\leq x-1 $. Therefore, in any case, the lowest segment $\Delta_{\rho}[x+1,y+1]$ commutes with $\rho\abs^{x}$ and $\rho\abs^{x+1}$. Let $u'$ be the essentially generalized Speh representation such that 
\[ \widetilde{u}_\Delta \hookrightarrow \Delta_{\rho}[x+1,y+1] \times u'. \]
Then  
\begin{equation}\label{eq: strip off}
    \begin{aligned}
     \widetilde{u}_\Delta \times (\rho\abs^{x+1})^{2s} \times (\rho\abs^{x})^{2t} &=  (\rho\abs^{x+1})^{2s} \times (\rho\abs^{x})^{2t} \times  \widetilde{u}_\Delta\\
    &\hookrightarrow (\rho\abs^{x+1})^{2s} \times (\rho\abs^{x})^{2t} \times \Delta_{\rho}[x+1,y+1] \times u'\\
    &= (\rho\abs^{x+1})^{2s} \times \Delta_{\rho}[x+1,y+1] \times (\rho\abs^{x})^{2t}  \times u'\\
    &\hookrightarrow (\rho\abs^{x+1})^{2s} \times \rho\abs^{x+1} \times \Delta_{\rho}[x,y+1] \times (\rho\abs^{x})^{2t}  \times u'\\
    &=(\rho\abs^{x+1})^{2s+1} \times (\rho\abs^{x})^{2t} \times \Delta_{\rho}[x,y+1]   \times u'\\
    &\hookrightarrow (\rho\abs^{x+1})^{2s+1} \times (\rho\abs^{x})^{2t+1} \times \Delta_{\rho}[x-1,y+1]   \times u'.
\end{aligned}
\end{equation}
Now combining this embedding and \eqref{eq I2 empty (ii)}, Part (a) follows from Frobenius reciprocity.

For Part (b), recall from \eqref{eq I2 nonempty (ii) tau} that
\[ D^{(2s+1)}_{\rho\abs^{x+1}}(\tau)= u_{\Delta} \times \bigtimes_{i \in I^{\neq x+1}} u_i \times \bigtimes_{i \in I^{=x+1}} L_{\rho\abs^{x+1}} \circ R_{\rho\abs^{-x-1}}(u_i) \rtimes \pmb{1}_{\Sp_{0}}.\]
Now we compute the order of the highest derivatives of the above product with respect to $\rho\abs^x$:
\begin{itemize}
    \item the term $u_{\Delta}$ will contribute $2$ to the order of the highest derivatives;
    \item each term $u_i$ in the product $\bigtimes_{i \in I^{\neq x+1}} u_i $ will contribute $2$ to the order of the highest derivatives if $x_i=x$;
    \item each term $L_{\rho\abs^{x+1}} \circ R_{\rho\abs^{-x-1}}(u_i)$ in the product $\bigtimes_{i \in I^{=x+1}} L_{\rho\abs^{x+1}} \circ R_{\rho\abs^{-x-1}}(u_i)$ will contribute $2$ to the order of the highest derivative if $y_i \leq x$.
\end{itemize}
We conclude that the order of the highest derivatives of $D^{(2s+1)}_{\rho\abs^{x+1}}(\tau)$ with respect to $\rho\abs^x$ is exactly
\[2+2|\{i\in I\mid\rho_i\cong\rho,\ x_i=x+1,\ y_i\leq x\}|+2|\{i\in I\mid\rho_i\cong\rho,\ x_i=x\}| =2t. \]
This completes the proof of Part (b).
\end{proof}

Finally, we deal with the remaining case: $G=\Sp_{2n}(F)$, $\rho\cong\one_{\GL_1}$, $\pi_0=\abs^{1/2}\rtimes\one_{\Sp_0}$, $\Delta=\abs^{-3/2}$ and $I_2=\emptyset$. In this case, we still first  take the derivatives $D^{(2s+1)}_{\abs^{x+1}}=D^{(2s+1)}_{\abs^{-1/2}}$, but then we take derivatives with respect to $\rho \abs^{1/2}= \rho\abs^{x+2}$.
\begin{prop}\label{prop:empty-32}
Suppose that $\pi$ is as in Theorem~\ref{thm:reduced} with $I_2=\emptyset$, that $G=\Sp_{2n}(F)$, $\pi_0=\abs^{1/2}\rtimes\one_{\Sp_0}$, $\rho\cong\one_{\GL_1}$, and that $\Delta=\abs^{-3/2}$. Let $s$ be as in Proposition~\ref{prop:empty} and
\[t:=|\{i\in I\mid\rho_i\cong\rho,\ x_i=-\tfrac12\}|+|\{i\in I\mid\rho_i\cong\rho,\ x_i=\tfrac12\}|.\]
\begin{enumerate}
 \item [(a)]If $\sigma\hookrightarrow \widetilde{u}_\Delta\rtimes\pi$, then $D^{(2t+2)}_{\rho\abs^{1/2}}\circ D^{(2s+1)}_{\rho\abs^{-1/2}}(\sigma)\neq0$. 
    \item [(b)] Let $\tau$ be defined in \eqref{eq tau}. Then $D^{(2t+2)}_{\rho\abs^{1/2}}\circ D^{(2s+1)}_{\rho\abs^{-1/2}}(\tau)=0$.
\end{enumerate}
In particular, $\pi$ is not unitary.
\end{prop}

\begin{proof}
   We continue the computation in Proposition \ref{prop:empty}. For Part (a), we have an injection
    \begin{align*}
        \pi \hookrightarrow (\rho\abs^{-1/2})^{2s} \times \left (\bigtimes_{i \in I^{=-1/2}} (\Delta_{\rho}[-3/2, y_i])^{2} \times (\abs^{-3/2})^2 \right) \rtimes \widetilde{\pi},
    \end{align*}
    where
    \[\widetilde{\pi}= \bigtimes_{i \in I^{=-1/2}} \widetilde{u}_i \times \bigtimes_{i \in I^{\neq -1/2}} u_i \rtimes \pi_{0}.\]
    Again, we compute the order of the highest derivative of $\widetilde{\pi}$ with respect to $\rho\abs^{1/2}$ by the Leibniz rule and Lemma \ref{lem:consecutive der Speh}. The key observation is that if $i \in I^{=-1/2}$, then $\widetilde{u}_i$ is non-trivial (i.e. $b_i\geq 3$), since otherwise $u_i= \widetilde{u}_\Delta,$ which contradicts the condition in Definition \ref{def:reduced}. Therefore, we have
\begin{itemize}
    \item each term $\widetilde{u}_i$ in the product $\bigtimes_{i \in I^{=-1/2}} \widetilde{u}_i $ will contribute $2$ to the order of the highest derivative;
    \item each term $u_i$ in the product $\bigtimes_{i \in I^{\neq -1/2}} u_i $ will contribute $2$ to the order of the highest derivative if $x_i=1/2$;
        \item the term $\pi_0$ will contribute $1$ to the order of the highest derivatives.
\end{itemize}
We conclude that $D_{\rho\abs^{1/2}}^{(2t+1)}(\widetilde{\pi})$ is the highest derivatives, and 
\begin{align*}
    \pi &\hookrightarrow  (\rho\abs^{-1/2})^{2s} \times \left (\bigtimes_{i \in I^{=-1/2}} (\Delta_{\rho}[-3/2, y_i])^{2} \times (\abs^{-3/2})^2 \right) \rtimes \widetilde{\pi}\\
    & \hookrightarrow (\rho\abs^{-1/2})^{2s} \times \left (\bigtimes_{i \in I^{=-1/2}} (\Delta_{\rho}[-3/2, y_i])^{2} \times (\abs^{-3/2})^2 \right) \times (\rho\abs^{1/2})^{2t+1}  \rtimes D_{\rho\abs^{1/2}}^{(2t+1)}(\widetilde{\pi})\\
    &= (\rho\abs^{-1/2})^{2s} \times (\rho\abs^{1/2})^{2t+1} \times \left (\bigtimes_{i \in I^{=-1/2}} (\Delta_{\rho}[-3/2, y_i])^{2} \times (\abs^{-3/2})^2 \right)   \rtimes D_{\rho\abs^{1/2}}^{(2t+1)}(\widetilde{\pi})\\
    &= (\rho\abs^{-1/2})^{2s} \times (\rho\abs^{1/2})^{2t+1} \rtimes \Pi,
\end{align*}
where $\Pi$ is the remaining part. Now if $\sigma\hookrightarrow \widetilde{u}_\Delta\rtimes\pi$, we have 
\begin{align*}
    \sigma &\hookrightarrow u_{\rho}(1,2) \times (\rho\abs^{-1/2})^{2s} \times (\rho\abs^{1/2})^{2t+1} \rtimes \Pi\\
    &= (\rho\abs^{-1/2})^{2s} \times u_{\rho}(1,2) \times (\rho\abs^{1/2})^{2t+1} \rtimes \Pi\\
    & \hookrightarrow (\rho\abs^{-1/2})^{2s+1}  \times (\rho\abs^{1/2})^{2t+2} \rtimes \Pi.
\end{align*}
Thus, Part (a) follows from Frobenius reciprocity.

For Part (b), recall from \eqref{eq I2 nonempty (ii) tau} that
\[ D^{(2s+1)}_{\rho\abs^{-1/2}}(\tau)= u_{\Delta} \times \bigtimes_{i \in I^{\neq -1/2}} u_i \times \bigtimes_{i \in I^{=-1/2}} L_{\rho\abs^{-1/2}} \circ R_{\rho\abs^{1/2}}(u_i) \rtimes \pmb{1}_{\Sp_{0}}.\]
For $i \in I^{=-1/2}$, since $b_i\geq 3$ as observed above, we have 
\[ u_i':= L_{\rho\abs^{1/2}} \circ R_{\rho\abs^{-1/2}}( u_{\rho}(a_i,b_i-2)) \neq 0.\] 
Then,
\begin{align*}
    u_i &\hookrightarrow \Delta_{\rho}[-1/2,y_i] \times u_{\rho}(a_i, b_i-2) \times \Delta_{\rho}[-y_i,1/2]\\
    & \hookrightarrow \rho\abs^{-1/2} \times \Delta_{\rho}[-3/2,y_i] \times \rho\abs^{1/2}\times u_i' \times \rho\abs^{-1/2}\times \Delta_{\rho}[-y_i,3/2] \times \rho\abs^{1/2}\\
    &= \rho\abs^{-1/2} \times \rho\abs^{1/2} \times \Delta_{\rho}[-3/2,y_i] \times u_i' \times \Delta_{\rho}[-y_i,3/2] \times \rho\abs^{-1/2} \times \rho\abs^{1/2}
\end{align*}
Thus, by Frobenius reciprocity, we have
\begin{align}\label{eq empty (iii)}
    (L_{\rho\abs^{1/2}} \circ R_{\rho\abs^{-1/2}}) \circ (L_{\rho\abs^{-1/2}} \circ R_{\rho\abs^{1/2}})  (u_i)\neq 0. 
\end{align}
On the other hand, look at
\[  u_i \hookrightarrow \Delta_{\rho}[-1/2,y_i] \times u_{\rho}(a_i, b_i-2) \times \Delta_{\rho}[-y_i,1/2]. \]
By Leibniz rule and Lemma \ref{lem:consecutive der Speh}, it follows that $(L_{\rho\abs^{1/2}} \circ R_{\rho\abs^{-1/2}}) \circ (L_{\rho\abs^{-1/2}} \circ R_{\rho\abs^{1/2}}) $ is a consecutive highest derivatives of the right hand side. Therefore, \eqref{eq empty (iii)} is also a consecutive highest derivatives.

Now we compute the order of the highest derivatives of $D_{\rho\abs^{-1/2}}^{(2s+1)}(\tau)$ with respect to $\rho\abs^{1/2}$:
\begin{itemize}
    \item the term $u_{\Delta}$ will not contribute to the order of the highest derivatives;
    \item each term $u_i$ in the product $\bigtimes_{i \in I^{\neq -1/2}} u_i$ will contribute $2$ to the order of the highest derivative if $x_i =1/2$;
    \item each term $ (L_{\rho\abs^{-1/2}} \circ R_{\rho\abs^{1/2}})  (u_i)$ in the product $\bigtimes_{i \in I^{=-1/2}} L_{\rho\abs^{-1/2}} \circ R_{\rho\abs^{1/2}}(u_i) $ will contribute $2$ to the order of the highest derivative by \eqref{eq empty (iii)}.
\end{itemize}
We conclude that the order of the highest derivatives of $D^{(2s+1)}_{\rho\abs^{-1/2}}(\tau)$ with respect to $\rho\abs^{1/2}$ is exactly
\[2|\{i\in I\mid\rho_i\cong\rho,\ x_i=1/2\}|+2|\{i\in I\mid\rho_i\cong\rho,\ x_i=-1/2\}| =2t. \]
This completes the proof of Part (b).    
\end{proof}

\subsection{The case \texorpdfstring{$I_2\neq\emptyset$}{I2 nonempty}}

We now deal with the remaining case that $I_2$ is non-empty. Set
\[ z:= \min_{i \in I_2}\{ y_i \}=\min_{i \in I_2}\{ 1-\half{a_i+b_i} \}.\]
 A key inequality, coming from our assumption that $u_i\not\cong \widetilde{u}_\Delta$ for all $i$, is
\begin{equation}\label{eq:key}
y+1<z.
\end{equation}

Now we describe the derivatives we will consider in this case. For $z-1 \leq w \leq x+1$, set 
\[ I_1^{w}=\{i \in I_1 \ | \ \rho_i\cong\rho,\ b_i>1,\ x_i=w \}.\]
Note the condition $b_i>1$ in the definition of this set. Put
\[ I_1':= \bigsqcup_{x+1 \geq w \geq z-1} I_{1}^{w},\qquad I_{1}^{[x+1, w]}:= \bigsqcup_{x+1 \geq w' \geq w} I_{1}^{w'},\qquad I_{1}^{[w, z-1]}:= \bigsqcup_{w \geq w' \geq z-1} I_{1}^{w'}.  \]
Then, for $z-1\leq w\leq x+1$ define
\[
T_w:=\bigtimes_{i\in I_1^{w}}\Delta_i^2\quad(w\neq x),\qquad T_x:=\Delta^2\times\bigtimes_{i\in I_1^{x}}\Delta_i^2 .
\]
 Identifying a segment $\seg{\rho}{\alpha}{\beta}$ with the set $\{\rho\abs^{\alpha},\rho\abs^{\alpha-1},\dots,\rho\abs^{\beta}\}$, we consider the multiset
\[
\Omega:=\sum_{i\in I_1'}2\cdot\big(\Delta_i\cap\seg{\rho}{x+1}{z-1}\big)+2\cdot\seg{\rho}{x}{z-1}=\Big\{(\rho\abs^{x+1})^{m_{x+1}},(\rho\abs^{x})^{m_x},\dots,(\rho\abs^{z-1})^{m_{z-1}}\Big\},
\]
where $m_w$ denotes the multiplicity of $\rho\abs^w$ in $\Omega$. In other words, $m_w$ is the number of segments of $T_{x+1}\times\cdots\times T_{w}$ containing $\rho\abs^w$. The derivative we consider is 
\[D^{(m_{z-1}+1)}_{\rho\abs^{z-1}}\circ\cdots\circ D^{(m_{x+1}+1)}_{\rho\abs^{x+1}}.\]

\begin{prop}\label{prop:nonempty}
Suppose that $\pi$ is as in Theorem~\ref{thm:reduced} with $I_2\neq\emptyset$. Let $\sigma$ be any irreducible subrepresentation of $\widetilde{u}_\Delta\rtimes\pi$. With the above notation, the following hold.
\begin{enumerate}
    \item [(a)]There is an injection 
    \[
\pi\hookrightarrow T_{x+1}\times T_x\times\cdots\times T_{z-1}\times\Pi\rtimes\pi_{0,temp},
\]
for some representation $\Pi$ of a general linear group.
\item [(b)] The following derivative does not vanish:
\[
D^{(m_{z-1}+1)}_{\rho\abs^{z-1}}\circ\cdots\circ D^{(m_{x+1}+1)}_{\rho\abs^{x+1}}(\sigma)\neq0 .
\]
\end{enumerate}
\end{prop}
\begin{proof}
    For Part (a), decompose $I_1=I_1' \sqcup I_1''$. For $x+1 \geq w \geq z-2$, set $\pi_w$ to be the irreducible representation such that $\pi$ is the Langlands subquotient of 
    \[ T_{x+1} \times \cdots \times T_{w+1} \rtimes \pi_{w}.\]
    More precisely, we have $\pi_{x+1}= \pi$, and $\pi_x$, $\pi_w$ for $x-1 \geq w \geq z-2$ are characterized by (see Proposition \ref{prop:embed}) 
    \begin{itemize}
        \item $\pi_x \hookrightarrow \Delta \times \Delta \times \bigtimes_{i \in I_1^{x+1}} (\widetilde{u}_{i}^-)^2 \times \bigtimes_{i \in I_{1}^{[x,z-1]} \sqcup I_1'' \sqcup I_2} (u_i^-)^2 \rtimes \pi_{0,temp} ;$ 
       \item $\pi_w \hookrightarrow  \bigtimes_{i \in I_1^{[x+1,w+1]}} (\widetilde{u}_{i}^-)^2 \times \bigtimes_{i \in I_{1}^{[w ,z-1]} \sqcup I_1''\sqcup I_2} (u_i^-)^2 \rtimes \pi_{0,temp} .$
    \end{itemize}
    We apply induction on $x+1 \geq w \geq z-1$ to show that
    \begin{align*}
        \pi_{w} \hookrightarrow T_{w} \rtimes \pi_{w-1},
    \end{align*}
which implies Part (a).

First, we deal with the base case $w=x+1$. Note that  $\Delta_i$ commutes with $\Delta$ for all $i \in I_{1}^{x+1}$. Therefore, 
\begin{align*}
   T_{x+1} \rtimes \pi_{x} \hookrightarrow T_{x+1}  \times \Delta^2 \times  \bigtimes_{i \in I_1^{x+1}} (\widetilde{u}_i^-)^2 \times \bigtimes_{i \in I_{1}^{[x,z-1]} \sqcup I_1''\sqcup I_2} (u_i^-)^2 \rtimes \pi_{0,temp}
     \hookrightarrow M(\pi).
\end{align*}
Here, we apply Proposition \ref{prop:embed} to $\pi_x$ for the first injection, and to $\pi_{x+1}$ for the second. As a consequence, $\pi=\pi_{x+1}$ is the unique irreducible subrepresentation of $T_{x+1} \rtimes \pi_{x}$.

Similarly, the case that $x \geq w \geq z-1$ follows from the injections
\[  T_{w} \rtimes \pi_{w-1} \hookrightarrow T_{w}  \times  \bigtimes_{i \in I_1^{[x+1,w]}} (\widetilde{u}_i^-)^2 \times \bigtimes_{i \in I_{1}^{[w-1,z-1]} \sqcup I_1''\sqcup I_2} (u_i^-)^2 \rtimes \pi_{0,temp} \hookrightarrow M(\pi_{w}).  \]
This proves Part (a).

For Part (b), we first show that 
\begin{align*}
    T_{x+1}\times T_x\times\cdots\times T_{z-1} \hookrightarrow (\rho \abs^{x+1})^{m_{x+1}}\times(\rho \abs^{x})^{m_{x}}\times\cdots\times(\rho \abs^{z-1})^{m_{z-1}} \times \Pi'
\end{align*}
for some representation $\Pi'$.  More precisely, for $x+1 \geq w \geq z-1$, set
\[ \Pi^{w}:= (\Delta_{\rho}[\min(x,w),y])^2 \times \bigtimes_{i \in I_1'} (\Delta_{\rho}[\min(x_i,w), y_i])^2.  \]
Here, the order of the product is taken so that $\min(x_i,w)$ is non-increasing. Also, if $w < y_i$, then $\Delta_{\rho}[w, y_i]$ means the trivial representation of $\GL_0$. We claim that
\[ T_{x+1}\times T_x\times\cdots\times T_{z-1} \hookrightarrow (\rho \abs^{x+1})^{m_{x+1}}\times(\rho \abs^{x})^{m_{x}}\times\cdots\times(\rho \abs^{w})^{m_{w}}  \times \Pi^{w-1} \]
for all $x+1 \geq w \geq z-1$. Indeed, when $w=x+1$, we have
\begin{align*}
    T_{x+1}\times T_x\times\cdots\times T_{z-1} &\hookrightarrow (\rho \abs^{x+1})^{m_{x+1}} \times \bigtimes_{i \in I_1^{x+1}} (\Delta_{\rho}[x,y_i])^2 \times  T_x\times\cdots\times T_{z-1}\\ &=(\rho \abs^{x+1})^{m_{x+1}} \times \Pi^{x}.
\end{align*}
For $x+1 >w \geq z-1$, we have
\begin{align*}
    \Pi^{w}& =  \left((\Delta_{\rho}[w,y])^2 \times\bigtimes_{i \in I_1^{[x+1,w]}} (\Delta_{\rho}[w, y_i])^2 \right) \times   \bigtimes_{i \in I^{[w-1,z-1]}_1} (\Delta_{\rho}[x_i,y_i])^2.
\end{align*}
Since the first part is an irreducible induction, with 
\[ L_{\rho\abs^w}^{(m_{w})}\left((\Delta_{\rho}[w,y])^2 \times\bigtimes_{i \in I_1^{[x+1,w]}} (\Delta_{\rho}[w, y_i])^2 \right)= (\Delta_{\rho}[w-1,y])^2 \times\bigtimes_{i \in I_1^{[x+1,w]}} (\Delta_{\rho}[w-1, y_i])^2 \]
being its highest derivative. Thus, Frobenius reciprocity gives
\[ \Pi^w \hookrightarrow (\rho\abs^w)^{m_w} \rtimes \Pi^{w-1}.\]
By induction hypothesis, this proves the claim.

Finally, since the lowest segment of $\widetilde{u}_\Delta$ is $\Delta_{\rho}[x+1, y+1]$, and $y+1 \leq z-1$ from \eqref{eq:key}, we see that $\widetilde{u}_\Delta$ commutes with $\rho\abs^{w}$ for all $x+1 \geq w \geq z-1$. Thus, repeating the argument in \eqref{eq: strip off}, we obtain an injection
    \begin{align*}
        \widetilde{u}_\Delta &\times (\rho \abs^{x+1})^{m_{x+1}}\times(\rho \abs^{x})^{m_{x}}\times\cdots\times(\rho \abs^{z-1})^{m_{z-1}}\\
        &\hookrightarrow (\rho \abs^{x+1})^{m_{x+1}+1}\times(\rho \abs^{x})^{m_{x}+1}\times\cdots\times(\rho \abs^{z-1})^{m_{z-1}+1}  \times \Delta_{\rho}[z-2,y+1] \times u',
    \end{align*}
    where $u'$ is the essentially generalized Speh representation such that 
    $\widetilde{u}_\Delta$ is the unique irreducible subrepresentation of $\Delta_{\rho}[x+1,y+1] \times u' .$
Combining with Part (a) and the above claim, this completes the proof of Part (b).
\end{proof}

\begin{prop}\label{prop:vanish}
In the setting of Proposition~\ref{prop:nonempty}, we have
\[
D^{(m_{z-1}+1)}_{\rho\abs^{z-1}}\circ\cdots\circ D^{(m_{x+1}+1)}_{\rho\abs^{x+1}}(\tau)=0 .
\]
\end{prop}

\begin{proof}
For simplicity, for tuples $\underline d=(d_{x+1},d_x,\dots,d_{z-1})$, $\underline l=(l_{x+1},\dots,l_{z-1})$ and\\ $\underline r=(r_{x+1},\dots,r_{z-1})$, set
\[
D_{(\underline d)}:=D^{(d_{z-1})}_{\rho\abs^{z-1}}\circ\cdots\circ D^{(d_{x+1})}_{\rho\abs^{x+1}},\quad
L_{(\underline l)}:=L^{(l_{z-1})}_{\rho\abs^{z-1}}\circ\cdots\circ L^{(l_{x+1})}_{\rho\abs^{x+1}},\quad
R_{(\underline r)}:=R^{(r_{z-1})}_{\rho\abs^{-z+1}}\circ\cdots\circ R^{(r_{x+1})}_{\rho\abs^{-x-1}}.
\]
Also, set 
\begin{align*}
    I^\circ:= \begin{cases}
        I \sqcup \{0\} & \text{ if }\pi_0=\pmb{1},\\
        I \sqcup \{0, -1\} & \text{ if }\pi_0= \abs^{1/2} \rtimes \pmb{1}_{\Sp_0},
    \end{cases}
\end{align*}
and put $u_0:= u_{\Delta}$, $u_{-1}:= \abs^{1/2}$. In this way, we have
\[ \tau= u_\Delta\times\bigtimes_{i\in I}u_i\rtimes\pi_0 = \bigtimes_{i \in I^\circ} u_i \rtimes \pmb{1}. \]

Now consider an arbitrary tuple $\underline d=(d_{x+1},d_x,\dots,d_{z-1})$. By Leibniz rule, we have  
\begin{align}\label{eq Leibniz final}
    D_{(\underline d)}(\tau)=\Big(\sum_{(\underline r^i,\underline l^i)}\ \bigtimes_{i\in I^\circ}R_{(\underline r^i)}\circ L_{(\underline l^i)}(u_i)\Big)\rtimes\one,
\end{align}
where the sum is taken over all collections $\{(\underline r^i,\underline l^i)\}_{i\in I^\circ}$ such that
\[
\sum_{i\in I^\circ}\big(r^i_w+l^i_w\big)=d_w\qquad\text{for all } z-1\leq w\leq x+1 .
\]
Based on Lemma \ref{lem:consecutive der Speh}, we now describe a necessary condition for the non-vanishing of \eqref{eq Leibniz final}.

For a multi-segment $\mathfrak m=\{\seg{\rho_k}{\alpha_k}{\beta_k}\}_{k\in K}$ (a multi-set of segments) and a supercuspidal $\rho\abs^w$, we define
\[
m_{\rho\abs^w}(\mathfrak m):=|\{k\in K\mid\rho\abs^w\in\seg{\rho_k}{\alpha_k}{\beta_k}\}|.
\]
For $\mathfrak m=\{\seg{\rho_k}{\alpha_k}{\beta_k}\}_{k\in K}$ and $\mathfrak n=\{\seg{\rho'_k}{\alpha'_k}{\beta'_k}\}_{k\in K'}$, we write $\mathfrak n\preceq\mathfrak m$ if there is an injection $\iota\colon K'\to K$ such that
\[
\seg{\rho'_k}{\alpha'_k}{\beta'_k}\subseteq\seg{\rho_{\iota(k)}}{\alpha_{\iota(k)}}{\beta_{\iota(k)}}\quad\text{and}\quad\alpha'_k=\alpha_{\iota(k)}\qquad\text{for all } k\in K'.
\]
We make three observations.
\begin{itemize}
\item Suppose that $R_{(\underline r^i)}\circ L_{(\underline l^i)}(u_i) \neq 0$ where $i \in I^{\circ} \setminus \{-1\}$. Then by Lemma \ref{lem:consecutive der Speh}, we have $0 \leq r^i_w \leq 1$ and $0 \leq l^i_w \leq 1$ for all $x+1 \geq w \geq z-1$. Moreover, the set $\{ \rho\abs^{w}\ | \ l^i_w \neq 0\}$ (resp.\ $\{ \rho\abs^{w}\ | \ r^i_w \neq 0\}$) forms a single segment $\Delta_L$ (resp. $\Delta_{R}$) such that
\[ \{\Delta_{L}, \Delta_{R}\} \preceq \{ \Delta_i, \Delta_i\}. \]
\item If $\mathfrak{n}_1 \preceq \mathfrak{m}_1$ and $\mathfrak{n}_2 \preceq \mathfrak{m}_2$, then the relation is preserved by disjoint union of multi-segments
\[ \mathfrak{n}_1  \sqcup \mathfrak{n}_2 \preceq \mathfrak{m}_1 \sqcup\mathfrak{m}_2.  \]
    \item If $\mathfrak n\preceq\mathfrak m$, then $m_{\rho\abs^w}(\mathfrak n)\leq m_{\rho\abs^w}(\mathfrak m)$ for all $w$.
\end{itemize}

Now suppose that $ D_{(\underline d)}(\tau) \neq 0$ for some tuple  $\underline d=(d_{x+1},d_x,\dots,d_{z-1})$. By the above observations and \eqref{eq Leibniz final}, there exists a multi-segment $\mathfrak{n}$ such that the following holds:
\begin{itemize}
    \item[(i)] We have $\mathfrak{n} \preceq \mathfrak{m}_{\pi}, $ where $\mathfrak{m}_{\pi}$ is the following multi-segment
    \[ \mathfrak{m}_{\pi}:= \{ \Delta_i, \Delta_{i}\}_{\substack{i \in I,\ \rho_i \cong \rho,\\ x+1 \geq x_i \geq z-1}} \sqcup \{ \Delta, \Delta \} \sqcup \{\abs^{1/2}, \abs^{-1/2}\}.\] 
    (Here the $\{\abs^{1/2}, \abs^{-1/2}\}$ comes from $u_{-1}$.)
    \item[(ii)] For all $x+1 \geq w \geq z-1$, we have $m_{\rho\abs^{w}}(\mathfrak{n})= d_w$. 
\end{itemize}
In particular, we have 
\[ d_{z-1} = m_{\rho\abs^{z-1}}(\mathfrak{n}) \leq m_{\rho\abs^{z-1}}(\mathfrak{m}_{\pi}). \]
However, recall that $ -1\geq  z-1 >y $. Thus,
\begin{align*}
     m_{\rho\abs^{z-1}}(\mathfrak{m}_{\pi})= 2 |\{i \in I\ | \ \rho_i \cong \rho,\ x+1 \geq x_i \geq z-1,\ y_i \leq z-1\}| +2,
\end{align*}
which is exactly $m_{z-1}$. We conclude that 
\[
D^{(m_{z-1}+1)}_{\rho\abs^{z-1}}\circ\cdots\circ D^{(m_{x+1}+1)}_{\rho\abs^{x+1}}(\tau)=0
\]
must hold, since the order of $D^{(m_{z-1}+1)}_{\rho\abs^{z-1}}$ is greater than $m_{\rho\abs^{z-1}}(\mathfrak{m}_{\pi})=m_{z-1}.$ This completes the proof of the proposition.
\end{proof}

\subsection{Proof of \texorpdfstring{Theorem~\ref{thm:reduced}}{Theorem 5.1}}
 We have shown in all cases (Propositions \ref{prop:empty}, \ref{prop:empty-half}, \ref{prop:empty-32}, \ref{prop:nonempty} and \ref{prop:vanish}) that the induction
 \[ \widetilde{u}_\Delta \rtimes \pi \]
 is not semisimple. Therefore, the representation $\pi$ is not unitary by \ref{U1}. \qed

\section{Proof of the main theorem}\label{sec:proof}

\begin{proof}[Proof of Theorem~\ref{thm:main}]
First, if $\pi$ satisfies condition $(\mathrm{A}')$, then $\pi$ is unitary by Proposition \ref{prop:A prime unitary}. For the other direction, suppose to the contrary that there exists a unitary representation $\pi$ of purely bad parity that does not satisfy condition $(\mathrm{A}')$. Then Proposition \ref{prop:reduction} produces a unitary minimal counter example $\pi'$, which contradicts Theorem \ref{thm:reduced}. This completes the proof of Theorem \ref{thm:main}.
\end{proof}


\end{document}